\documentclass[preprint,11pt,english]{imsart}
\usepackage[left=1in,right=1in,top=1in,bottom=1in]{geometry}

\RequirePackage[OT1]{fontenc}

\RequirePackage{multirow}

\RequirePackage{amsthm,amsmath,amssymb,mathrsfs,enumerate,mathtools}
\RequirePackage[numbers,sort]{natbib}
\RequirePackage[colorlinks,citecolor=blue,urlcolor=blue]{hyperref}
\RequirePackage{dsfont,float}
\RequirePackage[inline]{enumitem}

\usepackage{subfiles}

\newcommand{\dd}{\mathrm{d}}
\renewcommand{\v}[1]{\boldsymbol{#1}}
\newcommand{\m}[1]{\mathbf{#1}}

\newcommand{\cd}{\stackrel{\mathcal{D}}{\rightarrow}}

\newcommand{\cp}{\stackrel{\mathbb{P}}{\rightarrow}}
\newcommand{\cas}{\stackrel{\mathrm{a.s.}}{\rightarrow}}

\newcommand{\ind}{\mathds{1}}
\DeclareMathOperator*{\argmin}{arg\,min}

\renewcommand{\triangleq}{\coloneqq}

\theoremstyle{plain}
\newtheorem{theorem}{Theorem}
\newtheorem{lemma}{Lemma}
\newtheorem{proposition}{Proposition}
\newtheorem{assumption}{Assumption}
\newtheorem{corollary}{Corollary}

\theoremstyle{definition}
\newtheorem{algorithm}{Algorithm}
\newtheorem{definition}{Definition}

\theoremstyle{remark}

\newtheorem{example}{Example}

\begin{document}

\begin{frontmatter}
\title{The reproducing Stein kernel approach for post-hoc corrected sampling}
\runtitle{Reproducing Stein kernels}
\runauthor{Hodgkinson, Salomone, Roosta}

\begin{aug}

\author[A]{\fnms{Liam} \snm{Hodgkinson}\ead[label=e1]{liam.hodgkinson@berkeley.edu}},
\author[B]{\fnms{Robert} \snm{Salomone}\ead[label=e2]{r.salomone@unsw.edu.au}},
\and
\author[C]{\fnms{Fred} \snm{Roosta}\ead[label=e3]{fred.roosta@uq.edu.au}}
\address[A]{Department of Statistics, UC Berkeley, Berkeley, CA, 94720, USA.
\printead{e1}}
\address[B]{Centre for Data Science, School of Mathematical Sciences, Queensland University of Technology, 4000, Australia. \\ \printead{e2}}
\address[C]{School of Mathematics and Physics, University of Queensland, St Lucia, QLD, 4067, Australia. %
\printead{e3}}

\end{aug}

\begin{abstract}
Stein importance sampling \cite{Liu2017} 
is a widely applicable technique based on kernelized Stein discrepancy \cite{liu2016kernelized}, which corrects the output of approximate sampling algorithms by reweighting the empirical distribution of the samples. A general analysis of this technique is conducted for the previously unconsidered setting where samples are obtained via the simulation of a Markov chain, and applies to an arbitrary underlying Polish space. We prove that Stein importance sampling yields consistent estimators for quantities related to a target distribution of interest by using samples obtained from a geometrically ergodic Markov chain with a possibly unknown invariant measure that {\em differs} from the desired target. The approach is shown to be valid under conditions that are satisfied for a large number of unadjusted samplers, and is capable of retaining consistency when data subsampling is used. Along the way, a universal theory of reproducing Stein kernels is established, which enables the construction of kernelized Stein discrepancy on general Polish spaces, and provides sufficient conditions for kernels to be convergence determining on such spaces. These results are of independent interest for the development of future methodology based on kernelized Stein discrepancies.
\end{abstract}


\end{frontmatter}

\maketitle
\section{Introduction}
Our problem of interest is the efficient computation of integrals with respect to some target probability measure $\pi$. Adopting the Monte Carlo approach, an empirical distribution is used to approximate the target measure, formed from samples drawn according to it. However, in many problems of interest, it is not possible to simulate according to $\pi$ exactly, and so further approximate methods must be used. Arguably the most widely employed and general approach is Markov Chain Monte Carlo (MCMC): successively drawing samples as a realization of a Markov chain. The dominant approach to MCMC involves the simulation of a process that is $\pi$-ergodic, often constructed by the Metropolis-Hastings algorithm from an underlying irreducible and aperiodic Markov chain \cite{roberts2004general}. 

However, there has been significant recent interest in so-called {\em unadjusted} MCMC approaches \cite{dalalyan2019user,Durmus2017,ma2015complete,heek2019bayesian}. A common strategy with these methods is the approximate numerical simulation of Langevin diffusions, which are $\pi$-ergodic \cite{ma2015complete}. For the same computational effort, one can achieve substantially lower variance of estimates at the cost of incurring additional (asymptotic) bias. 
Despite poorer asymptotic guarantees \cite{dwivedi2019log}, the ensuing Markov chains are often rapidly mixing, and perform particularly well in high dimensional settings \cite{durmus2019high}. However, despite significant recent empirical success \cite{heek2019bayesian} and the presence of explicit error estimates \cite{dalalyan2019user,Durmus2017}, such algorithms have not been widely adopted. Statisticians and practitioners alike remain skeptical, as consistency is often considered a basic theoretical necessity for any sampling algorithm.

The aim of this work is to construct a universal theoretical framework for \emph{Stein importance sampling}. Consequently, one can achieve the best of both worlds by generating approximate samples from a rapidly mixing Markov chain, and then subsequently `correcting' them, without pre-existing knowledge of the chain, to obtain consistent estimators of quantities involving an unnormalized probability distribution. This allows us to obtain both a low-variance estimator with good practical performance, and a theoretical guarantee that increasing the number of samples will provide an estimator with diminishing error. The approach is viable for all algorithms, most notably, including biased samplers.
Stein importance sampling is not new, having been introduced as a form of black-box importance sampling in \cite{Liu2017}, but lacks sufficient theoretical justification to support its use as a universal post-hoc correction tool, particularly as it applies to Markov chains. It belongs to a large class of algorithms derived from the \emph{kernelized Stein discrepancy} (KSD), first introduced in \cite{liu2016kernelized}, and motivated by ideas in \cite{oates2017control}. Interpreted as an integral probability metric with respect to a fixed target distribution $\pi$, the KSD combines ideas from Stein's method \cite{stein1986approximate,ross2011fundamentals} and reproducing kernel Hilbert spaces \cite{aronszajn1950theory, steinwart2008support}, and has the advantage of being explicitly computable. KSD has been applied to a myriad of problems, including measurement of sample quality \cite{gorham2017measuring}, goodness-of-fit tests \cite{chwialkowski2016kernel,liu2016kernelized,yang2018goodness,kanagawa2019kernel}, variance reduction \cite{oates2017control,oates2019convergence,Liu2017}, variational inference \cite{liu2016two, pu2017vae}, and the construction of other sampling algorithms \cite{liu2016stein, wang2016learning, detommaso2018stein, chen2018stein, chen2019stein}. 

Stein importance sampling stands out as one of the most straightforward techniques involving KSD: for elements $X_1,\dots,X_n \in S$, weights $w_1,\dots,w_n$ are selected such that the resulting weighted empirical distribution $\hat{\pi}^n$ with integrals
\begin{equation}
\label{eq:WeightedAvg}
\hat{\pi}^n(\phi) = \sum_{i=1}^n w_i \phi(X_i),\qquad %
\phi\,:\, S \to \mathbb{R},
\end{equation}
is as close as possible to a target distribution $\pi$ under a corresponding KSD $\mathbb{S}(\, \cdot\,\Vert \pi)$. As all weights must be non-negative and sum to one for $\hat{\pi}^n$ to be a valid probability distribution, the resulting optimization problem reduces to the constrained quadratic program
\begin{equation}
\label{eq:WeightsOpt}
\v w = \argmin_{\v w \in \mathbb{R}^d} \mathbb{S}(\hat{\pi}^n \Vert \pi) = \argmin_{\v w \in \mathbb{R}^d}\left\{ {\v w}^{\top} \m K_\pi \v w \, : \, \v w \geq \v 0, \; \sum_{i=1}^n w_i = 1\right\},
\end{equation}
where $\m K_{\pi}$ is the Gram matrix corresponding to a \emph{reproducing Stein kernel} for $\pi$, where the inequality $\v w \geq \v 0$ is to be interpreted element-wise. 
Stein importance sampling
requires no modification to existing methods; it is a post-hoc correction applicable for any arbitrary collection $\{X_k\}_{k=1}^n$.
The program (\ref{eq:WeightsOpt}) is a well-studied problem for which many solution techniques exist, for example, interior point methods, or alternating direction method of multipliers (see \cite{nocedal2006numerical} for a comprehensive survey).

\subsection{Contribution}
In this work, we make several important theoretical contributions to the theory of Stein importance sampling and reproducing Stein kernels: 
\begin{enumerate}[label=(\Roman*),resume]
\item {Our main result is presented in \S\ref{sec:SIS} as Theorem \ref{thm:SteinCorrection}, where we establish sufficient conditions for Stein importance sampling --- when applied to samples generated by a Markov chain --- to yield consistent estimators for integrals of test functions under $\pi$.}
\end{enumerate}

The result is particularly interesting for two reasons. Firstly, somewhat surprisingly, we do not require that the Markov chain is $\pi$-ergodic.  In fact, technical conditions for the result that pertain to the ergodicity and convergence of the chain are easily verified for many unadjusted samplers, which are of increasing interest to practitioners due to their rapid mixing in high dimension. Secondly, the result is established for a generalized form of Stein importance sampling where certain terms are estimated unbiasedly, for example by data subsampling, which provides justification for the use of the method in the large data settings common in twenty-first century statistical computing.

To establish the above result for general underlying spaces, we first in \S\ref{sec:Stein} --- following some background exposition ---  establish new theory for Stein kernels and KSD that is of independent interest:
\begin{enumerate}[label=(\Roman*),resume]
\item Construction of reproducing Stein kernels, and therefore KSD, with respect to probability measures on \emph{arbitrary} Polish spaces is addressed.
\end{enumerate}
The construction subsumes existing approaches pertaining to Euclidean space \cite{oates2017control,liu2016kernelized}, discrete spaces \cite{yang2018goodness}, and Riemannian manifolds \cite{liu2018riemannian,barp2018riemannian}, and further enables construction of Stein kernels and KSD on any Hilbert space, and other separable function spaces, such as $\mathcal{C}([0,1])$. It builds off the generator approach for Stein's method \cite{barbour1988stein}; Proposition \ref{prop:SteinConstruct} shows that reproducing Stein kernels can always be constructed from the generators of ergodic processes with strongly continuous semigroups, and  Proposition \ref{prop:SteinBuild} deals with the construction of reproducing Stein kernels on product spaces.
Finally,
\begin{enumerate}[label=(\Roman*),resume]
\item We establish general sufficient conditions for reproducing Stein kernels to yield a convergence determining KSD.
\end{enumerate}
Previously, specialized direct proofs have been used to establish the convergence determining property for specific KSDs; see \cite{gorham2017measuring, chen2018stein}. In Propositions \ref{prop:SepImpliesConv} and \ref{prop:SepConditions}, we show that in the locally compact setting, separability and divergence of some function in the image of the RKHS under the Stein operator yields the convergence determining property.

\subsection{Notation}
In the sequel, $S$ denotes an arbitrary Polish space, with $\mathcal{C}(S)$ (resp. $\mathcal{C}_b(S)$) the set of arbitrary (resp. bounded) continuous functions from $S$ to $\mathbb{R}$, and $\mathcal{F}$ a Banach subspace of $\mathcal{C}(S)$. For locally compact $S$, $\mathcal{C}_0(S)$ denotes the closure of the space of compactly supported continuous functions from $S$ to $\mathbb{R}$ under the uniform norm $\|\phi\|_{\infty} = \sup_{x \in S}|\phi(x)|$. For $S \subset \mathbb{R}^d$, $\mathcal{C}^k(\mathbb{R}^d)$ denotes the set of $k$-times continuously-differentiable functions from $\mathbb{R}^d$ to $\mathbb{R}$. The space of probability measures on $S$ is denoted by $\mathcal{P}(S)$. The law of a random element $X$ is denoted $\mathcal{L}(X)$. Let $\mathcal{M}(S)$ denote the space of $\sigma$-finite measures on $S$, equipped with the vague topology. For any measure $\mu \in \mathcal{M}(S)$, the space of $p$-integrable functions with respect to $\mu$ is denoted $L^p(\mu)$. For each $\phi \in L^1(\mu)$, we let $\mu(\phi) \coloneqq \int_S \phi(x) \mu(\dd x)$. For $\nu \in \mathcal{M}(S)$, we write $\mu \ll \nu$ if $\nu$ is absolutely continuous with respect to $\mu$, that is, if $\mu(B) = 0$ implies $\nu(B) = 0$ for any Borel set $B$ --- the corresponding Radon-Nikodym derivative in such cases is denoted $\frac{\partial \nu}{\partial \mu}$. We write $\mu \equiv \nu$ when the measures $\mu$ and $\nu$ are equivalent, that is, if $\mu(B) = \nu(B)$ for any Borel set $B$. We also adopt big $\mathcal{O}$ notation in probability: $\{X_n\}_{n=1}^{\infty},\{Y_n\}_{n=1}^{\infty}$ satisfy $X_n = \mathcal{O}_{\mathbb{P}}(Y_n)$ if for any $\epsilon > 0$, there exists a finite $M > 0$ such that $\sup_n \mathbb{P}(|X_n / Y_n| > M) < \epsilon$. For any operator $\mathcal{A}$ and function $f$ in the domain of $\mathcal{A}$ (denoted $\mbox{dom}(\mathcal{A})$), $\mathcal{A}f(x)$ denotes the evaluation of the function $\mathcal{A}f$ at $x$. For operators $\mathcal{A}_1,\mathcal{A}_2$ acting on elements of a function space $\mathcal{F}$ of scalar-valued functions on $S$, $\mathcal{A}_1 \otimes \mathcal{A}_2$ denotes the tensor product of $\mathcal{A}_1$ and $\mathcal{A}_2$, acting on functions $f: S \times S \to \mathbb{R}$ by
\[
(\mathcal{A}_1 \otimes \mathcal{A}_2)f(x,y) = (\mathcal{A}_1 f_2(\cdot, y))(x),\quad\mbox{where}\quad f_2(x,y) = (\mathcal{A}_2 f(x, \cdot))(y).
\]
For function spaces $H_1$ and $H_2$, $H_1 \otimes H_2$ denotes the tensor product space of $H_1$ and $H_2$.
Finally, collections of elements are denoted in \textbf{bold}, and for any $\v x = (x_1,\dots,x_d)$ and $1 \leq i \leq d$, we denote $\v x_{-i} = (x_1,\dots,x_{i-1},x_{i+1},\dots,x_d)$.

\section{Reproducing Stein Kernels}
\label{sec:Stein}
Discrepancy measures are a fundamental concept in many areas of probability and statistics. However, 
the vast majority of discrepancies can not be estimated unbiasedly, let alone computed exactly. 
The majority of useful discrepancies valid for arbitrary probability measures lie in the
class of \emph{integral probability metrics} \cite{muller1997integral}, which compare two probability measures $\mu$ and $\pi$ in terms of their integrals with respect to some class of test functions $\Phi$:
\[
d_{\Phi}(\mu, \pi) = \sup_{\phi \in \Phi} |\mu(\phi) - \pi(\phi)|.
\]
For example, by taking $\Phi$ to be the unit ball in $\mathcal{C}(S)$, $d_{\Phi}$ becomes the total variation metric. Choosing $\Phi$ 
to be the set of Fr\'{e}chet-differentiable functions
with derivative bounded uniformly by one yields the 
Wasserstein-1 metric \cite[Remark 6.5]{villani2008optimal}. Conditions that guarantee $d_{\Phi}$ is a metric on the
space of probability measures, and that $d_{\Phi}(\mu_n, \pi) \to 0$ as $n\to\infty$
implies $\mu_n(\phi) \to \pi(\phi)$ for any bounded continuous function $\phi$ (the \emph{convergence determining}
property), can be found in \cite[Theorem 3.4.5]{ethier2009markov}.

There are now two obstacles encountered when computing a metric $d_{\Phi}$ between an empirical measure $\mu$ and a target probability measure $\pi$. Firstly, while the expectation $\mu(\phi)$ is easy enough to evaluate, 
$\pi(\phi)$ is not. In fact, the empirical measure $\mu$ is usually
designed to approximate $\pi$, and computing $\pi(\phi)$ is often the objective! One of the first universal methods for solving this problem is
\emph{Stein's method}, an ingenious technique which encodes details about a 
target probability measure $\pi$ into a so-called {\em Stein operator} $\mathcal{A}_{\pi}$. Nearly twenty years
after the method's introduction in \cite{stein1986approximate}, Barbour \cite{barbour1988stein} developed a popular framework
for constructing Stein operators known as the \emph{generator approach}, which
also serves as intuition behind the method.

We briefly outline the approach: let $\pi$ be an arbitrary probability measure on $S$, and let $X_t^{\pi}$ be a
$\pi$-ergodic Markov process that induces a strongly continuous semigroup $\{P_t\}_{t \geq 0}$ on $\mathcal{F}$ defined by $P_t f(x) = \mathbb{E}[f(X_t^\pi)\vert X_0^\pi = x]$ for $f \in \mathcal{F}$. Arguably the most important case is when $S$ is locally compact and $X_t^\pi$ is Feller, in which case $\{P_t\}_{t\geq 0}$ is strongly continuous on $\mathcal{C}_0(S)$ \cite[Theorem 19.6]{kallenberg2006foundations}. Let $\mathcal{A}_{\pi}$ denote the Markov generator of $X_t^\pi$, that is, the linear operator satisfying $\mathcal{A}_\pi f = \lim_{t \to 0^+} \frac1t(P_t f - f)$ for functions $f \in \mathcal{F}$ where the limit exists with respect to the topology of $\mathcal{F}$. Let $\Phi$ be a dense subset of the domain of $\mathcal{A}_\pi$ with respect to the topology of $\mathcal{F}$. By \cite[Proposition 1.1.5(c)]{ethier2009markov}, if $X_0^{\pi}$ is distributed according to $\mu$, then for any $\phi \in \Phi$ and $t \geq 0$,
\begin{equation*}
\mathbb{E}\, \phi(X_t^{\pi}) - \mu(\phi) = \int_0^t \mathbb{E}\mathcal{A}_{\pi} \phi(X_s^{\pi}) \dd s,
\end{equation*}
and by taking $t\to\infty$, since $X_t^{\pi} \cd \pi$,
\begin{equation}
\label{eq:SteinFTC}
\mu(\phi) - \pi(\phi) = -\int_0^{\infty}\mathbb{E}\mathcal{A}_{\pi} \phi(X_s^{\pi}) \dd s.
\end{equation}
For example, let $\mu$ and $\pi$ be Dirac measures at points $a,b \in \mathbb{R}^d$ (respectively) and $X_t^\pi = b + (a-b)e^{-t}$ a deterministic process satisfying $X_0^\pi = a$ and $\lim_{t\to\infty} X_t^\pi = b$. The infinitesimal generator of $X_t^\pi$ is $\mathcal{A}_\pi \phi(x) = (b - x) \cdot \nabla \phi(x)$, and so
$
\delta_{a}(\phi) - \delta_{b}(\phi) = -\int_0^{\infty} \nabla \phi(X_s^\pi) \cdot \dd X_s^\pi,
$
which is precisely the fundamental theorem of calculus for line integrals, along the path parameterized by $X_t^\pi$. 
Therefore, (\ref{eq:SteinFTC}) is a {\em stochastic} generalization of the fundamental
theorem of calculus for line integrals, where the infinitesimal generator plays the role of the gradient, and the stochastic process the curve joining two probability measures.

Keeping with this
analogy, the \emph{Stein equation} is a stochastic analogue of the mean value theorem. Let $\mathbb{E}_x$ denote the expectation operator subject to the event $\{X_0^\pi = x\}$. Under the same assumptions, by \cite[Proposition 1.1.5]{ethier2009markov}, the function $x \mapsto \int_0^t \mathbb{E}_x \phi(X_s^\pi) \dd s$ lies in the domain of $\mathcal{A}_\pi$ for any $\phi \in \Phi$ and $t \geq 0$, with image $x \mapsto \int_0^t \mathbb{E}_x\mathcal{A}_\pi \phi(X_s^\pi) \dd s$ (in other words, the generator and the integral may be exchanged). Because $\mathcal{A}_\pi$ is closed \cite[Proposition 1.1.6]{ethier2009markov}, assuming strong continuity of $\mathcal{A}_\pi$, one may derive the Stein equation
\begin{equation}
\label{eq:SteinEq}
\mu(\phi) - \pi(\phi) = \mu(\mathcal{A}_{\pi} f_\phi),
\end{equation}
where $f_\phi(x) = -\mathbb{E}_x\int_0^{\infty} [\phi(X_s^{\pi})-\pi(\phi)]\dd s$ (see the proof of \cite[Theorem 5]{gorham2019}, for example). Denoting by $\mathcal{F}_{\Phi}$ the image of $\Phi$ under the operation $\phi \mapsto f_\phi$, note
\begin{equation}
\label{eq:SteinMethod}
d_{\Phi}(\mu,\pi) = \sup_{f_\phi \in \mathcal{F}_{\Phi}} |\mu(\mathcal{A}_{\pi} f_\phi)|,
\end{equation}
no longer involves the expectation $\pi(\phi)$, and therefore sidesteps the challenge in its computation. An important observation of Stein
is that for any operator $\mathcal{A}_{\pi}$ with the property
\begin{equation}
\label{eq:SteinOperator}
\mathbb{E}\mathcal{A}_{\pi}\phi(X) = 0\mbox{ holds for all }\phi \in \Phi \mbox{ if and only if }\mathcal{L}(X) \equiv \pi,
\end{equation}
we may proceed to formulate (\ref{eq:SteinEq}), solve for $f_\phi$ accordingly,
and arrive at (\ref{eq:SteinMethod}), even if $\mathcal{A}_\pi$ is not a Markov generator. This procedure describes Stein's method at its highest level
of generality. Hence, operators satisfying (\ref{eq:SteinOperator}) are often called \emph{Stein operators}.

The second issue with computing $d_{\Phi}$ is the non-trivial optimization problem of obtaining the supremum over $\Phi$. In the context of (\ref{eq:SteinMethod}), this is
even more difficult due to the complex relationship between $\Phi$ and $\mathcal{F}_{\Phi}$. The trick, as outlined in \cite{liu2016kernelized}, is to
choose $\Phi$ such that $\mathcal{F}_{\Phi}$ is the unit ball of a reproducing kernel Hilbert space. Recall that a Hilbert space $H$ of functions $h:S\to\mathbb{R}$ is a \emph{reproducing kernel Hilbert space} (RKHS) if for all $x \in S$, the evaluation functional $h \mapsto h(x)$ is bounded \cite{aronszajn1950theory}. 
By the Riesz representation theorem, for any RKHS $H$, there exists a unique symmetric, positive-definite function $k:S\times S \to \mathbb{R}$, called
the \emph{reproducing kernel} of $H$, such that $k(x,\cdot) \in H$ for any $x \in S$, and
\begin{equation}
\label{eq:ReproducingProperty}
h(x) = \langle h, k(x,\cdot) \rangle_H,\qquad h \in H.
\end{equation}
The relation (\ref{eq:ReproducingProperty}) is often called the \emph{reproducing property} of the kernel. Conversely, the Moore-Aronszajn theorem \cite[Theorem 4.21]{steinwart2008support} states that any positive-definite kernel induces
a unique corresponding RKHS. There exists a significant theory of RKHS; see, for example, \cite{aronszajn1950theory} and \cite[\S4]{steinwart2008support}. Reproducing kernel Hilbert spaces arise in machine learning 
as they can be used to reduce infinite-dimensional
optimization problems over function spaces, to problems over finite-dimensional spaces.
Such a simplification occurs in (\ref{eq:SteinMethod}) when $\mathcal{F}_{\Phi}$ is the unit ball of a RKHS $H$, due to the existence of a \emph{reproducing Stein kernel}, which is guaranteed when the Stein operator $\mathcal{A}_{\pi}$ is the generator of a Markov
process (Proposition \ref{prop:SteinConstruct}). For brevity, we will often drop the term `reproducing' as it is clear in our context, but stress that these Stein kernels should not be confused with those seen in \cite{courtade2019existence}, for example.
\begin{proposition}[Stein Kernels from Generators]
\label{prop:SteinConstruct}
Let $\pi$ be a probability measure on $S$, and $X_t^\pi$ a $\pi$-ergodic process with strongly continuous semigroup $\{P_t\}_{t \geq 0}$ over $\mathcal{F}$ and Markov generator $\mathcal{A}_{\pi}$. Further, let $H \subseteq \mathcal{F}$ be a RKHS with reproducing kernel $k$. Assuming that $k(\cdot,x) \in \mathrm{dom}(\mathcal{A}_\pi)$ for any $x \in S$, 
then $H_{\pi} = \mathcal{A}_{\pi} H$ is a RKHS with reproducing kernel
\begin{equation}
\label{eq:SteinKernel}
k_{\pi}(x, y) = (\mathcal{A}_{\pi} \otimes \mathcal{A}_{\pi})k(x,y),\qquad x,y\in S.
\end{equation}
\end{proposition}
The proof relies on the interchangeability of inner products with Bochner integrals and differential operators on $\mathbb{R}^d$; see \cite[Lemma 4.34]{steinwart2008support} for example. This property extends to Markov generators through their time derivative formulation.
\begin{proof}%
Since the Markov generator is the time derivative of its corresponding semigroup, we can proceed in a similar fashion to \cite[Lemma 4.34]{steinwart2008support}
concerning interchangability of derivatives and inner products.
Let $P_t$ denote the semigroup of $X_t^{\pi}$, recalling that for any
$f \in \mathcal{F}$, 
\begin{equation}
    \label{eq:SemigroupIntegral}
P_t f ~=~\int_S f(x) \kappa_t(\cdot,\dd x),
\end{equation}
where
$\{\kappa_t\}_{t\geq 0}$ is the family of transition probability kernels of $X_t^{\pi}$. 
Interpreting (\ref{eq:SemigroupIntegral}) as a Bochner integral, for any $h \in H$, since $\langle h, \cdot \rangle_H$
is a continuous linear operator, the reproducing property together with
properties of the Bochner integral imply
\[
P_t h(x) = \int_S \langle h, k(y, \cdot) \rangle_H \; \kappa_t(x, \dd y) = \left\langle h, \int_S k(y, \cdot) \kappa_t(x, \dd y)\right\rangle_H.
\]
Therefore, by letting $\Delta_t = t^{-1}(P_t - I)$ for each $t > 0$, where $I$ is the identity operator,
\begin{equation}
\label{eq:SteinConDelta}
\Delta_t h(x) = \left\langle h, (\Delta_t \otimes I)k(x, \cdot) \right\rangle_H.
\end{equation}
To show that $(\mathcal{A}_{\pi} \otimes I) k$ exists, it suffices
to show that for any sequence of positive real numbers $\{t_n\}_{n=1}^{\infty}$
converging to zero and any $x \in S$, $\{(\Delta_{t_n}\otimes I)k(x,\cdot)\}_{n=1}^{\infty}$ is a Cauchy
sequence in $H$. Since $H$ is complete, by definition of the generator
$\mathcal{A}_{\pi}$,
\begin{equation}
\label{eq:SteinConGenDelta}
\lim_{n\to\infty} (\Delta_{t_n} \otimes I)k(x,\cdot) = (\mathcal{A}_{\pi}\otimes I)k(x,\cdot) \in H.
\end{equation}
For any $x \in S$, let $D_t(x) = (\Delta_t \otimes I)k(x, \cdot) \in H$. Choosing any $x \in S$ and $m,n\in \mathbb{N}$,
\begin{multline}
\label{eq:SteinConInnerExp}
\|D_{t_n}(x) - D_{t_m}(x)\|_H^2 
= \langle D_{t_n}(x), D_{t_n}(x)\rangle_H\\
+\langle D_{t_m}(x), D_{t_m}(x)\rangle_H
-2\langle D_{t_n}(x), D_{t_m}(x)\rangle_H.
\end{multline}
Now, for any $y \in S$, by multiple applications of (\ref{eq:SteinConDelta}),
\[
\langle D_{t_n}(x),D_{t_m}(y)\rangle_H
=
(\Delta_{t_n} \otimes \Delta_{t_m}) k(x,y).
\]
The Kolmogorov equations \cite[Proposition 1.5(b)]{ethier2009markov} state that for any $t \geq 0$ and $h$ in the domain of $\mathcal{A}_{\pi}$, $\Delta_t h(x) = t^{-1} \int_0^t P_s \mathcal{A}_{\pi} h(x) \dd s$, which, by the mean value theorem,
implies $\Delta_t h(x) = P_{\xi} \mathcal{A}_{\pi} h(x)$ for some $\xi \in (0,t)$. Since $k(\cdot,y)$ and $k(x,\cdot)$ are both assumed to be in the domain of
$\mathcal{A}_{\pi}$ (due to the symmetry of $k$), for some $\xi_{m,n} \in (0,t_n)$ and $\eta_{m,n} \in (0,t_m)$,
\[
(\Delta_{t_n}\otimes \Delta_{t_m})k(x,y) = (P_{\xi_{m,n}}\mathcal{A}_{\pi} \otimes
P_{\eta_{m,n}}\mathcal{A}_{\pi})k(x,y).
\]
The strong continuity of the semigroup $P_t$ implies that for any $\epsilon > 0$,
there is an $N \in \mathbb{N}$ such that for $m,n\geq N$,
\[
|\langle D_{t_n}(x), D_{t_m}(y)\rangle_H
- k_{\pi}(x,y)| < \epsilon.
\]
The above inequality combined with (\ref{eq:SteinConInnerExp}) implies that $\{D_{t_n}(x)\}_{n=1}^{\infty}$ is a Cauchy sequence, and so (\ref{eq:SteinConGenDelta}) holds.
Together with (\ref{eq:SteinConDelta}),
\[
\mathcal{A}_{\pi} h(x) = \langle h, (\mathcal{A}_{\pi} \otimes I)k(x, \cdot) \rangle_H.
\]
The result now follows from \cite[Theorem 4.21]{steinwart2008support}.
\end{proof}

\begin{example}[Riemannian Stein Kernels]
We rederive the construction of Stein kernels on Riemannian manifolds outlined in \cite{barp2018riemannian, liu2018riemannian}.
Let $(M,g)$ be a complete connected $d$-dimensional Riemannian manifold with basis coordinates $(\partial_j)_{j=1}^d$. The gradient of a smooth function $f:M\to\mathbb{R}$ is defined on $(M,g)$ by $\nabla f = \sum_{i,j=1}^d g^{ij} \frac{\partial f}{\partial x_i} \partial_j$, where $g^{ij}$ is the $(i,j)$-th element of the inverse of the matrix $\m G=(g_{ij})$ with elements $g_{ij} = g(\partial_i, \partial_j)$ for $i,j=1,\dots,d$. The \emph{Laplace-Beltrami operator} $\Delta$ on $(M,g)$ is the divergence of the gradient of $f$, given by
\[
\Delta f = \frac{1}{\sqrt{\det \m G}} \sum_{i,j=1}^d \frac{\partial}{\partial x_i} \left(g^{ij} \frac{\partial f}{\partial x_j}\sqrt{\det \m G} \right). 
\]

For any $\mathcal{C}^1$-smooth vector field $Z$ on $(M,g)$, there exists a unique Feller process $X_t$ on $(M,g)$ with generator $\Delta + Z$ \cite[\S2.1]{wang2014analysis}. If $\pi$ is absolutely continuous with respect to the Riemannian volume measure on $(M,g)$ with density $p$, the choice $Z = \nabla \log p$ yields a process $X_t$ with invariant measure $\pi$ \cite[pg. 72]{wang2014analysis}. Furthermore, under the Bakry-Emery curvature condition, the diffusion $X_t$ is ergodic, indeed, geometrically ergodic in the Wasserstein metric \cite[Theorem 2.3.3]{wang2014analysis}. Therefore, letting $\mathcal{A}_\pi = \Delta + \nabla \log p$ be the generator of the diffusion $X_t$, by Proposition \ref{prop:SteinConstruct}, $\mathcal{A}_\pi$ induces a Stein kernel for any base reproducing kernel $k$ on $(M,g)$, with target measure $\pi$. 
To complete the construction of Stein kernels on manifolds requires an underlying reproducing kernel. Any reproducing kernel $k$ on $\mathbb{R}^{2d}$ admits an (extrinsic) reproducing kernel on $(M,g)$ by the strong Whitney embedding theorem \cite[Theorem 2.2.a]{skopenkov2008embedding} ($(x,y) \mapsto k(\phi(x),\phi(y))$ where $x,y \in M$ and $\phi:M \to \mathbb{R}^{2d}$ is an embedding). Intrinsic kernels are known for special manifolds, such as the $d$-dimensional sphere $\mathbb{S}^d = \{x \in \mathbb{R}^{d+1}\; :\; \|x\| = 1\}$ \cite{alegria2018family}.
\end{example}

\begin{example}[Zanella-Stein Kernels]\label{Ex.Zanella}
Stein kernels can be constructed on arbitrary discrete structures using the \emph{Zanella processes} introduced in \cite{zanella2019informed}. Here, we follow the presentation of \cite{power2019accelerated}. Suppose that $\pi$ is a probability function on a countable set $\mathcal{X}$. For each $x \in \mathcal{X}$, specify a set $\partial x \subset \mathcal{X}$ as the \emph{neighbourhood} of $x$. Doing so endows $\mathcal{X}$ with a directed graph structure $G$ with edge set $\mathcal{E} = \{(x,y)\,:\,x \in \mathcal{X}, y \in \partial x\}$.
The choice of each neighbourhood is arbitrary, with the exception that the resulting graph $G$ must be strongly connected and aperiodic.
Let $g:\, \mathbb{R}_+ \to \mathbb{R}_+$ satisfy $g(x) = x g(1 / x)$ for all $x > 0$. Such functions are referred to as \emph{balancing functions}, and include $g(x) = \min\{1,x\}$ and $g(x) = \frac{x}{1+x}$. Then, the Markov jump process with infinitesimal generator
\[
\mathcal{A}_{\pi} f(x) = \sum_{y \in \partial x}  g\left(\frac{\pi(y)}{\pi(x)}\right) [f(y) - f(x)],
\]
satisfies the detailed balance equations on $\mathcal{X}$ with respect to $\pi$. Assuming the process is also ergodic, for any positive-definite \emph{graph kernel} $k$ on $G$ (see \cite{ghosh2018journey} for a review on the subject), applying Proposition \ref{prop:SteinConstruct} to $\mathcal{A}_\pi$ and $k$ yields a reproducing Stein kernel for $\pi$ on $\mathcal{X}$:
\[
k_\pi(x, y) = \sum_{x' \in \partial x,\, y' \in \partial y} g\left(\frac{\pi(x')}{\pi(x)}\right) g\left(\frac{\pi(y')}{\pi(y)}\right) [k(x',y') - k(x',y) - k(x,y') + k(x,y)].
\]

\end{example}

As a consequence of the proof of Proposition \ref{prop:SteinConstruct}, we require only that $\mathcal{A}_{\pi}h(x) = \langle h, (\mathcal{A}_{\pi}\otimes I) k(x,\cdot)\rangle_H$ for any $h \in H$ and $x \in S$, to infer that $H_\pi \subset H$ is a RKHS with reproducing kernel (\ref{eq:SteinKernel}). Therefore, for more general Stein operators, we rely on the following definition for constructing reproducing Stein kernels. 
\begin{definition}
A Stein operator $\mathcal{A}_\pi$ is said to \emph{induce a Stein kernel} if for any RKHS $H \subset \mathrm{dom}(\mathcal{A}_\pi)$ of scalar-valued functions on $S$ with reproducing kernel $k$,
\[
\mathcal{A}_\pi h(x) = \langle h, (\mathcal{A}_\pi \otimes I) k(x,\cdot)\rangle_H,\quad \mbox{for any }h \in H\mbox{ and }x \in S.
\]
\end{definition}

For any RKHS $H$ of functions on $S$,
target measure $\pi$, and operator $\mathcal{A}_{\pi}$ inducing a Stein kernel $k_{\pi}$,
the \emph{kernelized Stein discrepancy} (KSD) between $\mu$ and $\pi$, denoted $\mathbb{S}(\mu \Vert \pi)$, satisfies
\begin{align}
\label{eq:KSD}
\mathbb{S}(\mu\Vert \pi)^2 &\coloneqq \sup_{\substack{h \in H \\ \|h\|_H \leq 1}} |\mu(\mathcal{A}_{\pi} h)|^2
= \sup_{\substack{h \in H_\pi \\ \|h\|_{H_\pi}\leq 1}} |\mu(h)|^2 \nonumber\\
&= \sup_{\substack{h \in H_\pi \\ \|h\|_{H_\pi}\leq 1}} \left\langle h, \int_S k_\pi(x,\cdot) \mu(\dd x) \right\rangle_{H_\pi}^2
= \left\lVert \int_S k_\pi(x,\cdot) \mu(\dd x)\right\rVert_{H_\pi}^2 \nonumber
\\&= \int_{S \times S}  k_{\pi}(x,y) \mu(\dd x) \mu(\dd y),
\end{align}
where, as in Proposition \ref{prop:SteinConstruct}, $H_\pi = \mathcal{A}_\pi H$ and the second equality follows by the property $\|\mathcal{A}_\pi h\|_{H_\pi} = \|h\|_H$ for any $h \in H$.
This construction (\ref{eq:KSD}) resembles maximum mean discrepancy (MMD); see \cite[\S3.5]{muandet2017kernel} for further details.
Now, if $\mu$ is the empirical measure of samples $X_1,\dots,X_n$,
then
\begin{equation}
\label{eq:KSDEmpirical}
\mathbb{S}(\mu\Vert\pi)^2 = \frac1{n^2} \sum_{i,j=1}^n k_{\pi}(X_i,X_j).
\end{equation}

Both of the issues concerning computation of discrepancies for empirical distributions are now addressed, with an explicit and often easily computable formula for the evaluation of a particular integral probability metric. 
However, the kernelized Stein discrepancy is not limited to applications with empirical distributions, since, for any distribution $\mu$ on $S$, we can readily estimate
$\mathbb{S}(\mu\Vert \pi)^2 = \mathbb{E}\, k_{\pi}(X,X')$ where $X,X'$ are independently distributed according to $\mu$, using crude Monte Carlo. 

\subsection{Efficient construction of Stein kernels on product spaces}
Building upon Barbour's generator approach to Stein's method \cite{barbour1988stein}, using
Proposition \ref{prop:SteinConstruct} 
one can, in principle, construct Stein kernels over any desired space. 
Unfortunately, Stein kernels obtained in this way are not always efficient to compute, and do not always coincide with Stein kernels
currently in the literature. Consider the two following univariate cases:
\begin{itemize}[leftmargin=*]
\item $S = \mathbb{R}$:
If $\pi$ is an absolutely continuous probability measure on $\mathbb{R}$ with smooth density $p$ supported on a connected subset of $\mathbb{R}$, the overdamped
Langevin stochastic differential equation
\begin{equation}
\label{eq:LangevinSDE}
\dd X_t = \frac{p'(X_t)}{p(X_t)} \ \dd t + \sqrt{2} \dd W_t
\end{equation}
is an ergodic Feller process with stationary measure $\pi$. By It\^{o}'s lemma,
its Markov generator
$\mathcal{A}_{\pi}$ acts on functions $h \in \mathcal{C}^2(\mathbb{R})$ by
\begin{equation}
\label{eq:LangevinGen}
\mathcal{A}_{\pi} h(x) = h''(x) + \frac{p'(x)}{p(x)} \ h'(x),\qquad x \in \mathbb{R}.
\end{equation}
Because $h \in \mathcal{C}^2(\mathbb{R})$ is arbitrary, the extra derivatives on $h''$ and $h'$ are redundant for $\mathcal{A}_\pi$ to be a Stein operator,
and one may instead consider
\begin{equation}
\label{eq:SteinOpR}
\mathcal{A}_{\pi} h(x) = h'(x) + \frac{p'(x)}{p(x)} h(x),\qquad x \in \mathbb{R}.
\end{equation}
which coincides with the canonical Stein operator on $\mathbb{R}$ \cite{ley2017stein,oates2017control,stein1986approximate}.
The operator $\mathcal{A}_\pi$ as in (\ref{eq:SteinOpR}) is
no longer the generator of a Markov process, however it still induces a Stein
kernel nonetheless, as a consequence of \cite[Lemma 4.34]{steinwart2008support} (see also the
remark after Proposition \ref{prop:SteinConstruct}). 
Alternatively, in place of the
Langevin diffusion (\ref{eq:LangevinSDE}), one could start with any other
$\pi$-ergodic diffusion equation. See \cite{gorham2019} for other alternatives,
all of which induce Stein kernels under the reduction of derivatives, as in (\ref{eq:SteinOpR}).
\item $S \subseteq \mathbb{Z}$: If $\pi(\cdot)$ is
a probability mass function on $\mathbb{Z}$ (or some other countable set, in
which case, one need only apply a injection into $\mathbb{Z}$), any birth-death
chain on the integers with birth and death rates $\alpha(\cdot)$ and $\beta(\cdot)$
respectively, will have stationary measure $\pi$ provided it satisfies the
detailed balance relations
\[
\alpha(x) \pi(x) = \pi(x + 1)\beta(x+1),\qquad \mbox{for all }x \in S.
\]
An immediate choice is
\begin{equation}
\label{eq:SteinDiscreteRates}
\alpha(x) \propto \pi(x+1),\qquad \beta(x) \propto \pi(x - 1),\qquad x \in S.
\end{equation}
Notice that the proposed solution (\ref{eq:SteinDiscreteRates}) does not require that $\pi$ be normalized.
The generator of this chain acts on all functions $h:\mathbb{Z}\to\mathbb{R}$ by
\[
\mathcal{A}_{\pi}h(x) = \alpha(x)[h(x+1)-h(x)]+\beta(x)[h(x-1)-h(x)],\qquad x \in S.
\]
Analogously to how the generator in (\ref{eq:LangevinGen}) acted on the derivatives of $h$, $\mathcal{A}_{\pi} h$ is acting on the (backward) difference of $h$. A similar reduction yields
\begin{equation}
\label{eq:SteinDiscreteOp}
\mathcal{A}_{\pi} h(x) = \alpha(x) h(x+1) - \beta(x) h(x),\qquad x \in S.
\end{equation}
It can be readily shown that (\ref{eq:SteinDiscreteOp}) induces a Stein
kernel. The Stein operator (\ref{eq:SteinDiscreteOp}) is most common in discrete scenarios,
as seen in \cite{brown2001stein}, and coincides with the Stein-Chen operator
for the Poisson distribution when $\alpha(x) = 1$ and $\beta(x) = x+1$, for example. 
This was also the approach taken in \cite{yang2018goodness} for the construction of reproducing Stein kernels on discrete spaces. 
\end{itemize}

With (\ref{eq:SteinOpR}) as (\ref{eq:SteinDiscreteOp}) as building blocks, one can easily
construct Stein kernels on higher-dimensional spaces that are
products of $\mathbb{R}$ and $\mathbb{Z}$. The method for doing so
is a decomposition of the state space into a product of subspaces, and extends to products of arbitrary Polish spaces as well.
Let $S \subseteq S_1 \times \cdots \times S_d$ and $\pi$ be a distribution on
$S$. For each $i=1,\dots,d$, let $\pi_i(\cdot\vert \v x_{-i})$ denote the conditional distribution of $\pi$ over its $i$-th subspace $S_i$, for $\v x \in S$. If $\mathcal{A}_{\pi_i(\cdot\vert \v x_{-i})}$ is a Stein operator for $\pi_i(\cdot \vert \v x_{-i})$ for each $i=1,\dots,d$ and $\v x \in S$, then we may consider the operator
\[
\mathcal{A}_{\pi}\v f(\v x) = \mathcal{A}_{\pi_1(\cdot\vert \v x_{-1})} \mathcal{P}_1^{\v x} f_1(\v x) + \cdots + \mathcal{A}_{\pi_d(\cdot\vert \v x_{-d})} \mathcal{P}_d^{\v x} f_d(\v x),
\]
where $\v f = (f_1,\dots,f_d) : S \to \mathbb{R}^d$, and $\mathcal{P}_i^{\v x}$ is the projection operator mapping $f:S\to \mathbb{R}$ to the function $y \mapsto f(x_1,\dots,x_{i-1},y,x_{i+1},\dots,x_d)$ over $y \in S_i$. 
By linearity, if each $\mathcal{A}_{\pi_i(\cdot\vert \v x)}$ induces a Stein kernel,
then $\mathcal{A}_{\pi}$ also induces a Stein kernel. This is formalized
in Proposition \ref{prop:SteinBuild}, the proof of which builds upon the arguments of \cite[Theorem 1]{oates2017control}, and generalizes \cite[Theorem 1]{wang2017stein}.

\begin{proposition}[Stein Kernels on Product Spaces]
\label{prop:SteinBuild}
Suppose that $S \subseteq S_1 \times \cdots \times S_d$ and let $\pi$ be a probability measure on $S$. For each $i=1,\dots,d$,
let $\pi_i(\cdot \vert \v x_{-i})$ denote the conditional distribution of $\pi$ over its $i$-th subspace $S_i$. Furthermore,
for each $i$, let $\mathcal{A}_{\pi_i(\cdot\vert \v x_{-i})}$ be a Stein operator for $\pi_i(\cdot\vert \v x_{-i})$ acting on $\bigcup_{\v x \in S} \mathcal{P}_i^{\v x} H_i$, where $H_i$ is a RKHS of scalar-valued functions on $S$ with reproducing kernel $k_i$, that also induces a Stein kernel.
Then, the operator $\mathcal{A}_{\pi}$ acting on functions $\v h=(h_1,\dots,h_d) \in H_1 \otimes \cdots \otimes H_d$ as
\[
\mathcal{A}_{\pi} \v h(\v x) = \mathcal{A}_{\pi_1(\cdot\vert \v x_{-1})} \mathcal{P}^{\v x}_1 h_1(\v x) + \cdots + \mathcal{A}_{\pi_d(\cdot\vert \v x_{-d})} \mathcal{P}^{\v x}_d h_d(\v x),
\qquad x \in S,
\]
is a Stein operator for $\pi$ on $H = H_1 \otimes \cdots \otimes H_d$, and $\mathcal{A}_{\pi} H$ is a RKHS of scalar-valued functions on $S$ with reproducing kernel
\begin{equation}
\label{eq:SteinBuild}
k_{\pi}(\v x,\v y) = \sum_{i=1}^d (\mathcal{A}_{\pi_i(\cdot\vert \v x_{-i})}\mathcal{P}_i^{\v x} \otimes \mathcal{A}_{\pi_i(\cdot\vert \v y_{-i})}\mathcal{P}_i^{\v y}) k_i(\v x,\v y).
\end{equation}
\end{proposition}

\begin{proof}%
Since the zero function is in $H_i$ and $\mathcal{A}_{\pi_i(\cdot \vert \v x)}$ is a Stein operator for $\pi_i(\cdot \vert \v x)$ for each $i=1,\dots,d$, $\mathbb{E}\mathcal{A}_\pi \v h(\v X) = 0$ for all $\v h \in H$ if and only if the conditional distributions of $\v X$ are equivalent to those of $\pi$. Therefore, $\mathcal{A}_\pi$ is a Stein operator for $\pi$ on $H$. For each $i=1,\dots,d$ and $\v x \in S$, let $K_i(\v x) = (\mathcal{A}_{\pi_i(\cdot\vert \v x)} \mathcal{P}_i^{\v x} \otimes I)k_i(\v x, \cdot) \in H_i$, and let $\v K(\v x) = (K_i(\v x))_{i=1}^d$.
By the hypothesis, for any $\v h \in H$,
\[
\mathcal{A}_{\pi}\v h(x) = \sum_{i=1}^d \mathcal{A}_{\pi_i(\cdot\vert \v x)}\mathcal{P}_i^{\v x}h_i(\v x) = \sum_{i=1}^d \langle h_i, K_i(x)\rangle_{H_i} = \langle h, \v K(x) \rangle_{H}.
\]
Defining $H_{\pi}$ to be the space of functions $\v h: S \to \mathbb{R}^d$ such that the norm
\[
\|\v h\|_{H_{\pi}} \coloneqq \inf_{\v \phi \in H} \{\|\v \phi\|_{H} \mbox{ such that } \v h(\v x) = \langle \phi,
\v K(\v x)\rangle_{H} \mbox{ for all }\v x \in S\},
\]
is finite, \cite[Theorem 4.21]{steinwart2008support} implies that $H_{\pi}$ is the reproducing kernel Hilbert space with
reproducing kernel
\begin{align*}
k_{\pi}(\v x,\v y) &= \langle \v K(\v x), \v K(\v y)\rangle_{H} = \sum_{i=1}^d \langle \mathcal{A}_{\pi_i(\cdot\vert \v x)} \mathcal{P}_i^{\v x} K_i(\v x), \mathcal{A}_{\pi_i(\cdot\vert \v y)} \mathcal{P}_i^{\v y} K_i(\v y)\rangle_{H_i} \\
&= \sum_{i=1}^d (\mathcal{A}_{\pi_i(\cdot\vert \v x)}\mathcal{P}_i^{\v x}\otimes \mathcal{A}_{\pi_i(\cdot\vert \v y)}\mathcal{P}_i^{\v y})k_i(\v x,\v y),
\end{align*}
where the last line follows since each $\mathcal{A}_{\pi_i(\cdot\vert \v x)}\mathcal{P}_i^{\v x}$ commutes with the inner product. Since
$H_{\pi}$ can be seen to be the image of $H$ under $\mathcal{A}_{\pi}$, the result follows. 
\end{proof}

\begin{example}[Canonical Stein Kernel on $\mathbb{R}^d$]
As an especially important example, if $\pi$ admits a smooth density $p$ on $\mathbb{R}^d$, then
applying (\ref{eq:SteinOpR}), we can take
$\mathcal{A}_{\pi_i(\cdot \vert \v x)} \v h(\v x) = \partial_{x_i} h_i(\v x) + \partial_{x_i}(\log p(\v x)) h_i(\v x),$
and so
\begin{equation}
\label{eq:CanonSteinOp}
\mathcal{A}_{\pi} \v h(\v x) = \nabla \cdot \v h(\v x) + \nabla \log p(\v x) \cdot \v h(\v x),
\end{equation}
which is the well-known Stein operator for densities on $\mathbb{R}^d$. To build the corresponding
Stein kernel, expanding out
\begin{multline*}
(\mathcal{A}_{\pi_i(\cdot \vert \v x)}\otimes \mathcal{A}_{\pi_i(\cdot \vert \v y)})k(\v x,\v y) 
= \partial_{x_i} \partial_{y_i} k(\v x,\v y) + \partial_{x_i} \log p(\v x)
\partial_{y_i} k(\v x,\v y) \\
+ \partial_{y_i} \log p(\v y) \partial_{x_i} k(\v x,\v y) 
+ k(\v x,\v y)\partial_{x_i} \log\pi(\v x) \partial_{y_i} \log p(\v y).
\end{multline*}
Therefore, the Stein kernel $k_{\pi}$ as given by (\ref{eq:SteinBuild}) becomes
\begin{multline}
\label{eq:CanonStein}
k_{\pi}(\v x,\v y) = \nabla_{\v x} \cdot \nabla_{\v y} k(\v x,\v y) + \nabla \log p(\v x) \cdot
\nabla_{\v y} k(\v x,\v y) \\ +
\nabla \log p(\v y) \cdot \nabla_{\v x} k(\v x,\v y) + k(\v x,\v y) \nabla \log p(\v x)
\cdot \nabla \log p(\v y),
\end{multline}
which is precisely the canonical Stein kernel of \cite{oates2017control, liu2016kernelized, ley2017stein, gorham2017measuring}. 
There is an extensive literature available for choosing reproducing kernels on $\mathbb{R}^d$. The most common choices are {\em radial basis function} kernels of the form 
\begin{equation}
\label{eq:RadialBasis}
k(\v x,\v y) = \kappa(\|\v x - \v y\|^2),\qquad \mbox{where}\qquad\kappa:\mathbb{R}_+\to\mathbb{R}.
\end{equation}
By the Schoenberg construction, such a kernel has the required positive-definiteness property (and therefore induces a RKHS) over any dimension $d$, if and only if $\kappa \in \mathcal{C}[0,\infty) \cap \mathcal{C}^{\infty}(0,\infty)$ is completely monotone, that is, for every non-negative integer $m$, the sign of the $m$-th derivative $\kappa^{(m)}(x)$ of $\kappa$ is $(-1)^m$ for any $x > 0$ \cite[Theorem 7.13]{wendland2004scattered}. Examples of completely monotonic functions include $x \mapsto (\alpha + x)^{-\beta}$ (inducing the inverse multiquadric kernel) and $x \mapsto e^{-\alpha x}$ (inducing the Gaussian kernel) for any $\alpha,\beta > 0$. Furthermore, for smooth functions $f,g$, if $f$ and $g'$ are completely monotonic and $g:\mathbb{R}_+ \to \mathbb{R}_+$, then $f \circ g$ is completely monotonic \cite[Theorem 2]{miller2001completely}. Therefore, $x \mapsto (\alpha + \log(1 + x))^{-1}$ (inducing the inverse-log kernel) is a completely monotone function.
\end{example}
\begin{example}[Marginal Stein Kernel on $\mathbb{R}^d$] Letting $k$ denote a reproducing kernel on $\mathbb{R}$, for a RKHS $H$ of functions $h:\mathbb{R}\to\mathbb{R}$, we can extend $H$ to an RKHS $H_i$ of functions $h^i:\mathbb{R}^d \to \mathbb{R}$ by $h^i(x_1,\dots,x_d) = h(x_i)$ for $\v x \in \mathbb{R}^d$. Since $H$ and $H_i$ are isomorphic, $H_i$ is a RKHS on $\mathbb{R}^d$ with reproducing kernel $k_i$ satisfying $k_i(\v x, \v y) = k(x_i, y_i)$ for $\v x,\v y \in \mathbb{R}^d$. The Stein operator (\ref{eq:CanonSteinOp}) applied to (\ref{eq:SteinBuild}) with these choices of underlying kernels yields the \emph{marginal Stein kernel} first considered in \cite{wang2017stein}: %
\begin{multline}
\label{eq:KCCKernel}
k_\pi(\v x, \v y)
= \textstyle\sum_{i=1}^d [\partial_{x_i} \partial_{y_i} k(x_i,y_i) + \partial_{x_i} \log p(\v x)
\partial_{y_i} k(x_i,y_i) \\
+ \partial_{y_i} \log p(\v y) \partial_{x_i} k(x_i,y_i) 
+ k(x_i,y_i)\partial_{x_i} \log\pi(\v x) \partial_{y_i} \log p(\v y)].
\end{multline}
The kernel (\ref{eq:KCCKernel}) enjoys an increased sensitivity to single-dimensional perturbations over (\ref{eq:CanonStein}), which becomes especially apparent with improved performance when the dimension $d$ is large. On the other hand, it is unable to distinguish between distributions with equivalent marginals. %
\end{example}

\begin{example}[Stein Kernel on $\mathbb{Z}^d$]
Let $\pi(\cdot)$ be a probability mass function on $\mathbb{Z}^d$. Any reproducing kernel on $\mathbb{R}^d$ is also a reproducing kernel on $\mathbb{Z}^d$. Applying (\ref{eq:SteinDiscreteOp}) with the choice of transition rates (\ref{eq:SteinDiscreteRates}) to (\ref{eq:SteinBuild}) immediately yields
\begin{multline*}
k_{\pi}(x,y) = \textstyle\sum_{i=1}^d [\pi^{i}(x)\pi^{i}(y) k(x+e_i,y+e_i) 
- \pi^{-i}(x)\pi^{i}(y) k(x,y+e_i)\\
- \pi^{i}(x)\pi^{-i}(y) k(x+e_i,y)
+ \pi^{-i}(x)\pi^{-i}(y) k(x,y)],
\end{multline*}
where $\pi^{\pm i}(x) = \pi(x \pm e_i)$ and $e_i$ denotes the $i$-th unit coordinate vector. In cases where $\pi(x)$ decreases rapidly in $x$, this choice of Stein kernel can succumb to significant numerical error. Instead, one would prefer to involve ratios of $\pi$, such as $r_i(x) = \pi(x - e_i) / \pi(x)$. Observe that, if $k$ is a positive-definite kernel on $\mathbb{Z}^d$, then letting $\pi^+(x) = \pi(x) + \ind\{\pi(x) = 0\}$,
\begin{equation}
\label{eq:DiscreteBaseKernel}
(x,y) \mapsto \frac{k(x,y)}{\pi^+(x)\pi^+(y)}
\end{equation}
is also a positive-definite (and therefore, reproducing) kernel on $\mathbb{Z}^d$. Applying (\ref{eq:SteinDiscreteOp}) with the choice (\ref{eq:SteinDiscreteRates}),
letting $\iota_i^x = \ind\{x + e_i \in \mbox{supp}(\pi)\}$, the Stein kernel becomes
\begin{multline}
k_{\pi}(x,y) = \textstyle\sum_{i=1}^d [\iota_i^x\iota_i^y k(x+e_i,y+e_i) 
- r_i(x)\iota_i^y k(x,y+e_i)\\
- r_i(y)\iota_i^x k(x+e_i,y)
+ r_i(x)r_i(y) k(x,y)].
\end{multline}
\end{example}

\subsection{Convergence determining properties}

One of the most important properties for a discrepancy measure is its ability to distinguish (or separate) a target measure from other distributions. We say that the KSD $\mathbb{S}(\cdot \Vert \pi)$ is \emph{separating} when
\[
\mathbb{S}(\mu \Vert \pi) = 0 \mbox{ if and only if } \mu \equiv \pi.
\]
A stronger condition for the KSD $\mathbb{S}(\cdot \Vert \pi)$ is to be \emph{convergence determining}:
\[
\mathbb{S}(\mu_n \Vert \pi) \to 0 \mbox{ implies } \mu_n \cd \pi.
\]
To our knowledge, \cite[Theorem 8]{gorham2017measuring}, \cite[Theorem 4]{chen2019stein}, and \cite[Theorems 3 \& 4]{chen2018stein} comprise the KSDs that are currently known to be convergence determining.
It is known that separation alone does not necessarily imply that a KSD is convergence determining for arbitrary sequences of probability measures \cite[Theorem 6]{gorham2017measuring}. However, separation is sufficient to guarantee that the KSD is convergence determining for tight sequences of probability measures: by \cite[Lemma 3.4.3]{ethier2009markov}, if $\{\mu_n\}_{n=1}^{\infty} \subset \mathcal{P}(S)$ is a tight sequence of probability measures and $\mathbb{S}(\cdot \Vert \pi)$ is separating, then $\mathbb{S}(\mu_n \Vert \pi) \to 0$ if and only if $\mu_n$ converges weakly to $\pi$. In the case where $S$ is locally compact, an extension involving random probability measures is presented in Proposition \ref{prop:SepImpliesConv}\ref{enu:SepTight}. For a separating KSD to be convergence determining for \emph{arbitrary} sequences of probability measures, it suffices to show that the KSD can \emph{detect tightness}, that is, if $\mathbb{S}(\mu_n \Vert \pi) \to 0$, then $\{\mu_n\}_{n=1}^{\infty}$ is tight. For this, the following assumption is sufficient.
\begin{assumption}
\label{ass:KernelExplode}
For a continuous reproducing Stein kernel $k_\pi$ on a locally compact space $S$ with corresponding RKHS $H_\pi$, there exists $h \in H_\pi$ such that, for any $M > 0$, there is a compact set $K \subseteq S$ satisfying $\inf_{x \notin K} h(x) > M$.
\end{assumption}
Here, we adopt the convention that the infimum over the empty set is infinite, thus Assumption \ref{ass:KernelExplode} is satisfied if $S$ is compact. 

\begin{proposition}[Convergence Determination]
\label{prop:SepImpliesConv}
Suppose that $\mathbb{S}(\cdot \Vert \pi)$ is separating and $S$ is locally compact. Let $(\Omega, \mathcal{E}, \mathbb{P})$ be an underlying probability space and suppose that $\{\mu_n\}_{n=1}^{\infty}$ is a sequence of random probability measures, that is, $\mu_n: \Omega \to \mathcal{P}(S)$ for $n=1,2,\dots$. Then $\mathbb{S}(\mu_n \Vert \pi) \cp 0$ implies $\mu_n \cd \pi$ if either
\begin{enumerate*}[label=(\alph*)]
\item \label{enu:SepTight}$\{\mu_n(\omega)\}_{n=1}^{\infty}$ is tight for almost every $\omega \in \Omega$, or
\item \label{enu:SepConv}Assumption \ref{ass:KernelExplode} holds.
\end{enumerate*}
\end{proposition}

\begin{proof}%
Assume that \ref{enu:SepTight} holds. For any Borel set $B \subseteq S$ and $\omega \in \Omega$, $\{\mu_n(\omega, B)\}_{n=1}^{\infty}$ is bounded in $\mathbb{R}_+$, and so $\{\mu_n\}_{n=1}^{\infty}$ is a tight sequence of random elements in $\mathcal{M}(S)$ \cite[Lemma 16.15]{kallenberg2006foundations}. Consider a weakly convergent subsequence $\mu_{n_k} \cd \mu$. Passing through to a further subsequence if necessary, $\mathbb{S}(\mu_{n_k}\Vert \pi) \to 0$ with probability one \cite[Lemma 4.2]{kallenberg2006foundations}. From \cite[Lemma 3.4.3]{ethier2009markov} and Prohorov's Theorem \cite[Theorem 16.3]{kallenberg2006foundations}, if $\mathbb{S}(\cdot\Vert\pi)$ is separating and $\mathbb{S}(\nu_n \Vert \pi) \to 0$ for a sequence of tight probability measures $\{\nu_n\}_{n=1}^{\infty}$, then $\nu_n \to \pi$ weakly. Therefore, since there exists a set of probability one on which $\mathbb{S}(\mu_{n_k}\Vert \pi) \to 0$ and $\{\mu_{n_k}(\omega)\}_{k=1}^{\infty}$ is a tight sequence of probability measures, $\mu_{n_k}(\omega) \to \pi$ weakly with probability one, implying that $\mu \equiv \pi$. Since $\{\mu_n\}_{n=1}^{\infty}$ was an arbitrary weakly convergent subsequence, $\mu_n \cd \pi$. 

Now suppose that \ref{enu:SepConv} holds. Let $\{\mu_{n_j}\}_{j=1}^{\infty}$ be an arbitrary subsequence of $\{\mu_n\}_{n=1}^{\infty}$. Since $\mathcal{M}(S)$ is Polish \cite[Theorem A2.3]{kallenberg2006foundations}, if we can show that there exists a further subsequence of $\{\mu_{n_j}\}_{j=1}^{\infty}$ that converges in distribution to $\pi$, then $\mu_n \cd \pi$. Since $\mathbb{S}(\mu_{n_j} \Vert \pi) \cp 0$ as $j\to \infty$, $\mu_{n_j}(h) \cp 0$, where $h \in H_\pi$ has the property described in Assumption \ref{ass:KernelExplode}. Passing through to a subsequence if necessary, we may assume that $\mu_{n_j}(h) \cas 0$, and consequently, $\sup_j \mu_{n_j}(h) < +\infty$ with probability one. Since $k_\pi$ is continuous, $h$ is also continuous (see \cite[pg. 345]{aronszajn1950theory}), and since there exists a compact set $N$ such that $\inf_{x \notin N} h(x) > 0$, $h$ is bounded below by $-C$ where $C \geq 0$. Therefore, the function $\tilde{h} = h + C$ is non-negative, and $\sup_j \mu_{n_j}(\tilde{h}) < +\infty$ with probability one. For $\epsilon > 0$, there exists a compact set $K$ such that $\inf_{x \notin K} \tilde{h}(x)~>~\epsilon^{-1} \sup_j \mu_{n_j}(\tilde{h})$ and so with probability one,
\[
\sup_j \mu_{n_j}(S \setminus K) \leq \left(\inf_{x\notin K} \tilde{h}(x)\right)^{-1}\sup_j \mu_{n_j}(\tilde{h}) < \epsilon.
\]
Since $\epsilon$ was arbitrary, $\{\mu_{n_j}(\omega)\}_{j=1}^{\infty}$ is tight for almost every $\omega \in \Omega$, and $\mu_{n_j} \cd \pi$ by \ref{enu:SepTight}. Therefore, $\mu_n \cd \pi$.

\end{proof}

Proposition \ref{prop:SepImpliesConv} suggests a general strategy for identifying Stein kernels that induce convergence determining KSD via Assumption \ref{ass:KernelExplode}. 
In particular, for $S = \mathbb{R}^d$, to satisfy Assumption \ref{ass:KernelExplode}, one should first look for Stein kernels such that $|k_\pi(\v x,\v y)| \to +\infty$ as $\|\v x\| \to \infty$ for fixed $\v y \in \mathbb{R}^d$. As noted in \cite{gorham2017measuring}, for the canonical Stein kernel, this requires that the base kernel $k$ diminishes sufficiently slowly relative to the growth of $\nabla \log p$. For this reason, the inverse multiquadric (IMQ) kernel $k(\v x, \v y) = (\alpha + \|\v x-\v y\|^2)^{-\beta}$ for $\alpha > 0$, $\beta < \frac12$, is a good choice for subgaussian distributions, inducing a convergence-determining KSD \cite[Theorem 8]{gorham2017measuring}. In Example \ref{ex:IMQScore}, we show directly that the IMQ score kernel satisfies Assumption \ref{ass:KernelExplode}. 

\begin{example}[IMQ Score Kernel]
\label{ex:IMQScore}
Consider $k_\pi$ the canonical Stein kernel (\ref{eq:CanonStein}) on $\mathbb{R}^d$ constructed from the IMQ score kernel \cite{chen2018stein}
\[
k_{\text{IMQS}}(\v x, \v y) = (\alpha^2 + \|\nabla f(\v x) - \nabla f(\v y)\|^2)^{-\beta},
\]
where $\alpha,\beta > 0$, and $f = -\log p$. Assume that $\nabla f$ is surjective and $M$-Lipschitz continuous, that is, there exists $M > 0$ such that $\|\nabla f(\v x)-\nabla f(\v y)\| \leq M\|\v x-\v y\|$ for all $\v x, \v y \in \mathbb{R}^d$. Furthermore, assume that $|k_\pi(\v x, \v y)| \to \infty$ as $\|\v x\| \to \infty$ for any $\v y \in \mathbb{R}^d$. For each $i=1,\dots,d$, let $\v y_i^+,\v y_i^-$ satisfy $\nabla f(\v y_i^+) = \v e_i$ and $\nabla f(\v y_i^-) = -\v e_i$, where $\v e_i$ is the $i$-th standard basis vector in $\mathbb{R}^d$. Letting \[K(\v x, \v y) = \nabla f(\v x) \cdot \nabla f(\v y) k_{\text{IMQS}}(\v x, \v y),\] we observe that $k_\pi(\v x, \v y) / K(\v x, \v y) \to 1$ as $\|\v x\|,\|\v y\| \to \infty$. Now, letting $\kappa(x) = (\alpha^2 + x)^{-\beta}$, since $\kappa$ is completely monotone,
\[
\sum_{i=1}^d K(\v x, \v y_i^+) + K(\v x, \v y_i^-) = \sum_{i=1}^d \partial_i f(\v x)[\kappa(\|\nabla f(\v x) - \v e_i\|^2) - \kappa(\|\nabla f(\v x) + \v e_i\|^2)] \to +\infty
\]
as $\|\v x\| \to \infty$. The same is also true for the function $\v x \mapsto \sum_{i=1}^d [k_\pi(\v x, \v y_i^+) + k_\pi(\v x, \v y_i^-)]$, which lies in $H_\pi$. Therefore, $k_\pi$ satisfies Assumption \ref{ass:KernelExplode}. 
\end{example}

Our remaining challenge now is finding sufficient conditions for $\mathbb{S}$ to be separating. Assuming that $S$ is locally compact, a reproducing kernel $k$ is called \emph{$\mathcal{C}_0$-universal} if its corresponding RKHS $H$ of functions $h:S\to \mathbb{R}$ is dense in $\mathcal{C}_0(S)$ and $k(x,\cdot) \in \mathcal{C}_0(S)$ for all $x \in S$ \cite{sriperumbudur2011universality}. Any RKHS over $\mathbb{R}^d$ with a non-constant reproducing kernel of the radial basis function form (\ref{eq:RadialBasis}) is $\mathcal{C}_0$-universal if $\kappa \in \mathcal{C}_0(S)$ \cite[Theorem 17]{micchelli2006universal}. Moreover, any restriction of a $\mathcal{C}_0$-universal kernel on $\mathbb{R}^d$ to a subset (for example, $\mathbb{Z}^d$), is also $\mathcal{C}_0$-universal on that restricted space. Every KSD arising from a $\mathcal{C}_0(S)$-universal kernel on $\mathbb{R}^d$ is separating for smooth densities from a certain restricted class of probability distributions \cite[Theorem 2.2]{chwialkowski2016kernel}. It is important that we can show it can separate $\pi$ from \emph{every} other probability measure. 

A convenient condition, inspired by \cite{glynn1996liapounov, gorham2019}, is presented in Proposition \ref{prop:SepConditions}. Recall that a Markov process $X_t$ is \emph{exponentially ergodic} if there exists some $C,\rho > 0$ such that for any bounded measurable function $f$,
\begin{equation}
\label{eq:ExpErgodic}
\|P_t f - \pi(f)\|_{\infty} \leq C \|f\|_{\infty} e^{-\rho t},\qquad t \geq 0.
\end{equation}
Necessary and sufficient conditions for exponential ergodicity of Markov processes are available in \cite{down1995exponential}. For overdamped Langevin diffusion targeting a measure $\pi$ that admits a smooth density $p$, that is, $X_t^\pi$ satisfying
\[
\dd X_t^\pi = \nabla \log p(X_t^\pi) \dd t + \sqrt{2} \dd W_t,
\]
which is the multivariate generalization of (\ref{eq:LangevinSDE}), if $\pi \in \mathcal{C}_0^2(S)$ and there exists some $\gamma \in (0,1)$ such that
\[
\liminf_{\|x\|\to\infty} (1-\gamma)\|\nabla \log p(x)\|^2 + \Delta \log p(x) > 0,
\]
then $X_t^\pi$ is exponentially ergodic \cite[Theorem 2.3]{roberts1996}. On the other hand, for finite case $|S| < \infty$, by combining \cite[Theorem 5.3]{down1995exponential} and \cite[Theorem 4.9]{levin2017markov}, we can see that \emph{any} irreducible finite-state Markov jump process is exponentially ergodic. 

\begin{proposition}[Separation]
\label{prop:SepConditions}
Assume that $S$ is locally compact. The kernelized Stein discrepancy $\mathbb{S}(\cdot \Vert \pi)$ corresponding to the reproducing Stein kernel $k_{\pi}$ in Proposition \ref{prop:SteinConstruct} is separating if $\mathcal{A}_\pi$ is the generator of an exponentially ergodic Markov process and $k$ is $\mathcal{C}_0$-universal.
\end{proposition}
\begin{proof}%
By construction, $\mathbb{S}(\mu \Vert \pi) = 0$ if $\mu \equiv \pi$, so assume that $\mathbb{S}(\mu \Vert \pi) = 0$, or in other words, $\mu(\mathcal{A}_\pi h) = 0$ for any $h \in H$ (the RKHS corresponding to $k$). Because $k$ is $\mathcal{C}_0$-universal, to show that $\mu \equiv \pi$, it will suffice to show that $\mu(h) = \pi(h)$ for any function $h \in H$ (see, for example, \cite[pg. 51]{muandet2017kernel}). It further suffices to show that $\mu(h) = 0$ for any $h \in H$ satisfying $\pi(h) = 0$. Let $P_t = e^{t \mathcal{A}_\pi}$ be the semigroup corresponding to $\mathcal{A}_\pi$, which, by assumption, satisfies (\ref{eq:ExpErgodic}). Consider the resolvent operators $\{R_\lambda\}_{\lambda > 0}$ defined for $f \in \mathcal{C}_0(S)$ by
\[
R_\lambda f = \int_0^{\infty} e^{-\lambda t} P_t f \; \dd t,\qquad \lambda > 0,
\]
which satisfy for any $\lambda > 0$, $(\lambda - \mathcal{A}_{\pi})R_\lambda = I$, where $I$ is the identity operator (see \cite[Theorem 19.4]{kallenberg2006foundations}). Since $k$ is $\mathcal{C}_0$-universal, $H \subset \mathcal{C}_0(S)$ and by the properties of the Bochner integral, $\|R_\lambda h\|_H \leq \lambda^{-1} \|P_t h\|_H \leq \lambda^{-1} \|h\|_H$, implying $R_\lambda\;:\;H \to H$. Choosing some fixed $h \in H$ with $\pi(h) = 0$, together with the hypotheses, this implies that $\mu(\mathcal{A}_\pi R_\lambda h) = 0$ and so
\begin{align}
|\mu(h)| &= |\mu((\lambda - \mathcal{A}_\pi) R_\lambda h)|,\nonumber\\
&\leq \lambda |\mu(R_\lambda h)| + |\mu(\mathcal{A}_\pi R_\lambda h)| \leq \lambda \|R_\lambda h\|_\infty.\label{eq:MuResolvent}
\end{align}
Now, since $\pi$ is the stationary distribution of $\{P_t\}_{t \geq 0}$ and $\pi(h) = 0$, $\pi(P_t h) = 0$ for any $t \geq 0$. Therefore, since $P_t$ is geometrically ergodic and $h$ is bounded,
\[
\|R_\lambda h\|_{\infty} \leq \int_0^\infty e^{-\lambda t}\|P_t h - \pi(h)\|_{\infty} \dd t \leq C \|h\|_\infty \int_0^\infty e^{-(\lambda + \rho) t} \dd t = \frac{C \|h\|_{\infty}}{\lambda + \rho},
\]
and, in particular, $\lambda \|R_\lambda h\|_{\infty} \to 0$ as $\lambda \to 0^+$. Together with (\ref{eq:MuResolvent}), this implies that $\mu(h) = 0$, and since $h \in H$ was arbitrary, the result follows.
\end{proof}

Proposition \ref{prop:SepConditions} can be extended to kernels constructed from Proposition \ref{prop:SteinBuild}, simply by applying the result to conditional probability measures. However, extending Proposition \ref{prop:SepConditions} to the canonical reproducing Stein kernel on $\mathbb{R}^d$ is more challenging, since the corresponding Stein operator differs from a Markov generator by an extra derivative. Fortunately, this special case has been treated in \cite{gorham2017measuring} using results from \cite{gorham2019}. In short, exponential convergence in the semigroup is still relevant, but required for Lipschitz-continuous, rather than bounded, test functions. Therefore, if the overdamped Langevin diffusion corresponding to $\pi$ mixes exponentially fast in Wasserstein distance, then the kernelized Stein discrepancy is separating. While exponential ergodicity of overdamped Langevin diffusion is known to hold for light-tailed target densities, sufficient conditions on exponential convergence in the Wasserstein metric often require log-concavity of the target density $\pi$, or some other distant dissipativity condition \cite{cattiaux2014semi, gorham2019}. In particular, the KSD for the canonical Stein kernel on $\mathbb{R}^d$ is separating if $\pi$ is distantly dissipative \cite{gorham2017measuring}.
From \cite[eq. 12]{hansen2003geometric}, distant dissipativity also implies exponential ergodicity of overdamped Langevin diffusion.

\section{Stein Importance Sampling}
\label{sec:SIS}
Importance sampling is a classic Monte Carlo technique \cite{kroese2013handbook}: for any (possibly unnormalized) target probability measure $\pi$, if $X_1,\dots,X_n$ are independent and identically distributed samples from a probability measure $\nu$ with $\pi \ll \nu$, then for $\hat{\pi}^n$ defined by (\ref{eq:WeightedAvg}),
\begin{equation}
\label{eq:ImportanceSampling}
\hat{\pi}^n(\phi) \cp \pi(\phi) \qquad \mbox{if} \qquad w_i = \frac{\partial\pi /\partial \nu(X_i)}{\sum_{i=1}^n \partial \pi / \partial \nu(X_i)}.
\end{equation}
By the ergodic theorem, the same is true if $X_1,\dots,X_n$ come from a $\nu$-ergodic Markov chain. 
Furthermore, in both cases, it can be shown that the rate of convergence as $n\to \infty$ is $\mathcal{O}_{\mathbb{P}}(n^{-1/2})$, for any $\phi \in \mathcal{C}_0(S)$. 
Unfortunately, for this choice of weights, $\hat{\pi}^n(\phi)$ will often have prohibitively large variance unless $\pi$ and $\nu$ are close in total variation, and this is further exacerbated in high dimensions. On the other hand, in cases where $\nu$ is constructed to be a close approximation for $\pi$, such as when $X_1,\dots,X_n$ are derived from some approximate sampling algorithm, explicit or computable forms for $\nu$ are rarely available.

These issues can be avoided entirely using Stein importance sampling \cite{Liu2017}. 
Recall that,
for a reproducing Stein kernel constructed with respect to a target measure $\pi$, the kernelized Stein discrepancy between a weighted empirical measure (\ref{eq:WeightedAvg}) and $\pi$ is
\begin{equation}
\label{eq:WeightedKSD}
\mathbb{S}(\hat{\pi}^n \Vert \pi) = \sum_{i,j=1}^n w_i w_j k_{\pi}(X_i, X_j)
= \v w^{\top} \m K_{\pi} \v w,
\end{equation}
where $\m K_{\pi} = (k_{\pi}(x_i, x_j))_{i,j=1}^n$ is the Gram matrix associated to $k_{\pi}$ and the points $X_1,\dots,X_n$. To obtain estimators of the form (\ref{eq:WeightedAvg}) that most closely match a given target measure, Stein importance sampling involves choosing weights $\v w = (w_1,\dots,w_n)$ that minimize the KSD (\ref{eq:WeightedKSD}), that is, the solution to the constrained quadratic program
\begin{equation}
\label{eq:BBISQuadProg}
\argmin_{\v w}\left\lbrace \v w^\top \m K_{\pi} \v w \, : \, \sum_{i=1}^n w_i = 1,\quad \v w \geq \v 0 \right\rbrace.
\end{equation}
There are two major advantages to this: (i) the underlying density of the samples $X_1,\dots,X_n$ is no longer needed, and (ii) the mean squared error of the corresponding estimator is often smaller than that obtained using classical importance sampling. %
One drawback of Stein importance sampling is the cubic computational complexity in the number of samples, however, many practitioners may not seek any more than a few thousand effective samples, for which solving (\ref{eq:BBISQuadProg}) is typically inexpensive.

\subsection{Stein kernel estimators}
If computing a Stein kernel is either expensive or intractable, one might instead seek to conduct Stein importance sampling with only a collection of estimators of a Stein kernel. To be valid, it is necessary that the Gram matrix formed from these kernels remains positive-semidefinite with high probability. To cover this condition and facilitate further theoretical analysis, we extend the definition of \emph{Gram-de Finetti matrices}, introduced in \cite{dovbysh1984gram}, to the kernel setting (Definition \ref{def:GramDeFinetti}).
\begin{definition}[Gram-de Finetti Kernel Arrays]
\label{def:GramDeFinetti}
An infinite array of random kernels $(k_{ij})_{i,j=1}^{\infty}$ on a probability space $(\Omega, \mathcal{E}, \mathbb{P})$ with $k_{ij}(\omega):S\times S \to \mathbb{R}$ for each $i,j=1,2,\dots$, is \emph{Gram-de Finetti} if for every sequence $\{x_i\}_{i=1}^{\infty} \subset S$, $(k_{ij}(x_i,x_j))_{i,j=1}^{\infty}$ is a Gram-de Finetti matrix, that is, a symmetric jointly exchangeable array ($k_{ij}(x_i,x_j)$ is equivalent in distribution to $k_{\sigma(i)\sigma(j)}(x_{\sigma(i)},x_{\sigma(j)})$ for any finite permutation $\sigma$ of $\mathbb{N}$) that is almost surely positive-semidefinite. 
\end{definition}
Observe that, for any positive-definite kernel $k$, 
$(k)_{i,j=1}^\infty$ is a Gram-de Finetti array.
Like Gram-de Finetti matrices, such arrays possess  convenient representations as functions over uniformly distributed random variables. The following Lemma \ref{lem:GramDeFinetti} is established without difficulty by following the same arguments seen in proofs of the Dovbysh-Sudakov representation theorem (see \cite[Lemma 1]{panchenko2010dovbysh}).
\begin{lemma}
\label{lem:GramDeFinetti}
Let $(k_\pi^{ij})_{i,j=1}^{\infty}$ be a Gram-de Finetti array of kernels. There exists a function $K_\pi$ and independent random variables $\xi$, $(\xi_i)_{i=1}^{\infty}$ uniformly distributed in $[0,1]$ such that
\[
k_\pi^{ij}(x,y) = K_\pi(x,y,\xi,\xi_i,\xi_j).
\]
\end{lemma}

\begin{example}[Subsampled Stein Kernels]
Consider the setting where a density $p$ decomposes as
\[
p(x) = \prod_{i = 1}^n p_i(x),\qquad x\in S,
\]
Here, each $p_i$ may correspond to a single data point, in which case, $n$ denotes
the total number of data points used to construct the density $p$. Rather than
computing $p$ directly, it is popular to utilise \emph{subsampling}, where
$\log p$ is estimated by
\[
\log {p}^{\mathscr{S}}(x) = \frac{n}{|\mathscr{S}|}\sum_{i \in \mathscr{S}} \log p_i(x),\qquad x\in S,
\]
where $\mathscr{S} \subseteq \{1,\dots,n\}$. Since $\log {p}^{\mathscr{S}}$ often involves far fewer terms than $\log p$ itself, it can be significantly faster to compute. Such estimates are common in subsampled MCMC, based upon the pseudomarginal approach  \cite{quiroz2019speeding}. One of the biggest challenges with these methods is that ${p}^{\mathscr{S}}$ is not an unbiased estimator of $p$, which the pseudomarginal approach requires. However, $\nabla \log {p}^{\mathscr{S}}$ \emph{is} an unbiased estimator of $\nabla \log p$. Consequently, the \emph{subsampled canonical Stein kernel}:
\begin{multline}
\label{eq:CanonSteinSubsample}
k_{\pi}^{ij}(x,y) = \nabla_x \cdot \nabla_y k(x,y) + \nabla \log {p}^{\mathscr{S}_i}(x) \cdot
\nabla_y k(x,y) \\ +
\nabla \log {p}^{\mathscr{S}_j}(y) \cdot \nabla_x k(x,y) + k(x,y) \nabla \log {p}^{\mathscr{S}_i}(x)
\cdot \nabla \log {p}^{\mathscr{S}_j}(y),
\end{multline}
is an unbiased estimator of the canonical Stein kernel (\ref{eq:CanonStein}) when $i\neq j$, provided that each $\mathscr{S}_i \subset \{1,\dots,n\}$ contains elements that are drawn independently with replacement, uniformly at random. Furthermore, for a sequence of such subsets $\mathscr{S}_1,\mathscr{S}_2,\ldots \subset \{1,\dots,n\}$, the array of subsampled Stein kernels $(k_{ij})_{i,j=1}^{\infty}$ is Gram-de Finetti.
\end{example}

Our proposed procedure for utilizing Stein kernel estimators in Stein importance sampling is presented in Algorithm \ref{alg:SteinIS}.

\begin{algorithm}[Stein Importance Sampling]
\label{alg:SteinIS}
Given a collection of samples $\{X_i\}_{i=1}^n$, test function $\phi$, and Gram-de Finetti array of kernels $(k_\pi^{ij})_{i,j=1}^{\infty}$ that estimate (unbiasedly) a reproducing Stein kernel $k_{\pi}$ when $i \neq j$, execute the following steps:
\begin{enumerate}
    \item Compute the Gram matrix $\m K_{\pi} = (k_{\pi}^{ij}(X_i, X_j))_{i,j=1}^n$.
\item Solve (\ref{eq:BBISQuadProg}) to obtain the importance
sampling weights $\v w$.
\item Output $\hat{\pi}^n(\phi)= \sum_{i=1}^n w_i \phi(X_i)$ as an estimator of $\pi(\phi)$.
\end{enumerate}
\end{algorithm}

The use of Stein kernel estimators, or \emph{random} Stein kernels, was recently considered for developing goodness-of-fit tests for latent variable models in \cite{kanagawa2019kernel}. These choices of Stein kernel estimators for latent variable models can be applied with Stein importance sampling via Algorithm \ref{alg:SteinIS}. 

\subsection{The Stein correction}

We now study the theoretical properties of Stein importance sampling as a post-hoc correction for MCMC output. In constrast to the Metropolis correction, we term this procedure the \emph{Stein correction}. The primary evidence for the use of the Stein correction is the consistency guarantee provided in Theorem \ref{thm:SteinCorrection}, a Markov chain analogue of \cite[Theorem 3.2]{Liu2017} with the further extension to arbitrary Gram-de Finetti array of kernels. 
Before stating the theorem, we recall the definition of $V$-uniform ergodicity \cite[\S16]{meyn2012markov}.
A $\nu$-ergodic Markov chain with Markov transition operator $\mathcal{P}$ on $\mathcal{S}$ is $V$-uniformly ergodic for a measurable function $V:\mathbb{R}^d\to [1,\infty)$ if
\[
\sup_{x \in \mathcal{S}} \sup_{|\phi|\leq V}\frac{|\mathcal{P}^n \phi(x) - \pi(\phi)|}{V(x)} \to 0,\qquad \mbox{as } n\to\infty.
\]
For a $V$-uniformly ergodic Markov chain, $V$ is often referred to as a \emph{Lyapunov drift function}, which provides an upper bound on the rate of decay of test functions $\phi$ for which the convergence of $\mathcal{P}^n \phi$ is uniform.
In the sequel, $\hat{\pi}^n$ denotes a weighted empirical distribution derived from Stein importance sampling (Algorithm \ref{alg:SteinIS}) applied with the choice of $\{X_i\}_{i=1}^n$, $\pi$, and $k_\pi$ clear from context. 

\begin{theorem}[Consistency of Stein Correction]
\label{thm:SteinCorrection}
Let $\{X_t\}_{t=1}^{\infty}$ be a $\nu$-ergodic Markov chain on $S$ that is $V$-uniformly ergodic
for some Lyapunov drift function $V\;:\; S \to [1,\infty)$.
Let $k_{\pi}$ be a reproducing Stein kernel with respect to some target distribution $\pi$ on $S$ which is absolutely continuous with respect to $\nu$, and inducing a kernelized Stein discrepancy $\mathbb{S}$. Furthermore, let $(k^{ij}_\pi)_{i,j=1}^{\infty}$ be a Gram-de Finetti array of kernels independent of $\{X_t\}_{t=1}^{\infty}$ such that $\mathbb{E} k^{ij}_\pi = k_\pi$ for $i\neq j$. %
\begin{enumerate}[label=(\alph*)]
\item \label{enu:Stein1} If there exists $0 < r < 1$ such that 
\begin{equation}
\label{eq:enuStein1}
\mathbb{E}\sup_{x \in S} \frac{k^{ii}_{\pi}(x, x)^2}{V(x)^{r}} < +\infty,
\end{equation}
then
\[
\mathbb{S}(\hat{\pi}^n \Vert \pi) \cp 0 \qquad \mbox{as } n\to \infty.
\]
\item \label{enu:Stein2} If there exists $0 < r < 1$ such that
\begin{equation}
\label{eq:enuStein2}
\mathbb{E}\sup_{x \in S} \frac{\partial \pi}{\partial \nu}(x)^4\frac{k_{\pi}^{ii}(x, x)^2}{V(x)^{r}} < +\infty,
\end{equation}
then
\[
\mathbb{S}(\hat{\pi}^n \Vert \pi) = \mathcal{O}_{\mathbb{P}}(n^{-1/2}),\qquad \mbox{as } n\to \infty.
\]
\end{enumerate}
In either case, if $S$ is locally compact and $k_\pi$ is separating and satisfies Assumption \ref{ass:KernelExplode}, then $\hat{\pi}^n \cd \pi$ as $n \to \infty$.
\end{theorem}
\begin{proof}
By Lemma \ref{lem:GramDeFinetti}, there exists a function $K_\pi$ and independent uniformly distributed random variables $\xi$, $(\xi_i)_{i=1}^{\infty}$ such that $k_\pi^{ij}(x,y) = K_\pi(x,y,\xi,\xi_i,\xi_j)$ for all $x,y \in S$.
To prove the desired convergence, our primary tool is the variance bound of
\cite{fort2012simple} for $U$-statistics of non-stationary Markov chains. 
Since for each $n$, $\v w^{(n)}$ minimizes $\mathbb{S}_n(\v w, \v x) \coloneqq \sum_{i,j=1}^n w_i w_j k_{\pi}(x_i, x_j)$ over $\v w$, it suffices to find a sequence of normalized reference weights $\v v^{(n)}$ 
such that $\mathbb{S}_n(\v v^{(n)}, \v x) \cp 0$ in case \ref{enu:Stein1} and
$\mathbb{S}_n(\v v^{(n)}, \v x) = \mathcal{O}_{\mathbb{P}}(n^{-1})$ in case
\ref{enu:Stein2}. The final statement follows as a consequence of Proposition \ref{prop:SepImpliesConv}. These weights are to be constructed in reference to the usual importance sampling weights (\ref{eq:ImportanceSampling}) obtained via the weight function $w = \partial \pi / \partial \nu$. Furthermore, we let $\mathbb{E}_{\v X}$ denote the conditional expectation operator conditioned on the entire Markov chain $\{X_t\}_{t=1}^{\infty}$, so that
$\mathbb{E}_{\v X}k_\pi^{ij}(X_i, X_j) = k_\pi(X_i, X_j)$ for any $i,j$.

Beginning with case \ref{enu:Stein2} first, by (\ref{eq:enuStein1}), there exists a square integrable function $F$ and a constant $C$ such that
\[
\frac{\partial \pi}{\partial \nu}(x)^4 k_{\pi}^{ij}(x,x)^2 \leq C V(x)^r [1 + F(\xi, \xi_i, \xi_j)]^2
\]
for each $i,j$ and $x \in S$. Positive-definiteness ensures that 
\[
k_\pi^{ij}(x_i, x_j) \leq \tfrac12[k_\pi^{ii}(x_i, x_i) + k_\pi^{jj}(x_j, x_j)],\quad i,j=1,2,\dots,
\]
and so we may take $F$ to satisfy
\begin{equation}
\label{eq:CauchyIneqF}
F(\xi, x, y) \leq \tfrac12 [F(\xi, x, x) + F(\xi, y, y)],\qquad \mbox{for any }x,y\in S.
\end{equation}
Consider the augmented Markov chain $\tilde{X}_t = (X_t, \xi_t)$ for $t = 0,1,2,\dots$ Under the hypotheses, since $\xi_i$ are independent, $\tilde{X}_t$ is $\tilde{V}$-uniformly ergodic where
\begin{equation}
\label{eq:NewLyapunov}
\tilde{V}(x,u) = V(x)[1 + F(\xi, u, u)]^{1/r}.
\end{equation}
Letting $\varphi((x,u),(y,v)) = w(x) w(y) K_\pi(x,y,\xi,u,v)$, since each $k^{ij}_\pi$ is an unbiased estimator of a Stein kernel, we see that 
\[
\int_{S \times S} \int_{[0,1]^2} \varphi((x,u),(y,v)) \ \dd u \dd v \ \pi(\dd x) \pi(\dd y) = 0,
\]
that is, $\varphi$ is a degenerate kernel. Furthermore, from (\ref{eq:CauchyIneqF}),
\begin{align*}
C_b &\coloneqq \sup_{x,y\in S}\sup_{u,v \in [0,1]} \frac{\partial \pi}{\partial \nu}(x) \frac{\partial \pi}{\partial \nu}(y) \frac{|K_\pi(x,y,\xi,u,v)|}{\tilde{V}(x,u)^r + \tilde{V}(y,v)^r} \\
&\leq \sup_{x \in S} \sup_{u \in [0,1]} \frac{\partial \pi}{\partial \nu}(x)^2 \frac{|K_\pi(x,x,\xi,u,u)|}{\tilde{V}(x,u)^r}
< +\infty,
\end{align*}
almost surely. Let $\v v^{(n)}$ denote the normalized reference weights $v_i^{(n)} \propto w(X_i)$. Observe that
\begin{align}
\mathbb{S}_n(\v v^{(n)}, \v X) &= \frac{n^{-2} \sum_{i,j=1}^n \mathbb{E}_{\v X}\varphi((X_i,\xi_i), (X_j,\xi_j))}{n^{-2}\sum_{i,j=1}^n w(X_i)w(X_j)} \nonumber \\ &= W_n^{-2} [(1-n^{-1}) \mathbb{E}_{\v X}U_n + n^{-1} \mathbb{E}_{\v X} V_n] \label{eq:SteinNDecomp}
\end{align}
where $U_n$ is the degenerate $U$-statistic
\[
U_n \coloneqq \frac{2}{n(n-1)} \sum_{1\leq i < j \leq n} \varphi((X_i,\xi_i), (X_j,\xi_j)),
\]
and
\[
V_n \coloneqq \frac1n \sum_{i=1}^n \varphi((X_i,\xi_i), (X_i,\xi_i)),\qquad W_n \coloneqq \frac1n\sum_{i=1}^n w(X_i).
\]
Since, by assumption,
$|\varphi((x,u), (x,u))|^2 \leq C \tilde{V}(x,u)$ for some constant $C > 0$, it follows
from \cite[Theorem 17.01]{meyn2012markov} that $\|V_n\|_2 = \mathcal{O}(n^{-1/2})$.
Furthermore, by applying \cite[Corollary 2.3]{fort2012simple}, $\|U_n\|_2 = \mathcal{O}(n^{-1})$. Finally, since $w \in L^1(\nu)$, \cite[Theorem 17.01]{meyn2012markov} implies that $W_n \cas 1$ as $n \to \infty$. Therefore, $\mathbb{S}_n(\v v^{(n)}, \v X) = \mathcal{O}_{\mathbb{P}}(n^{-1})$ as required. 

For case \ref{enu:Stein1}, the results of \cite{fort2012simple} can no longer
be applied directly. Instead, consider the truncated reference weights
$v_i^{(n)} \propto w(X_i) \wedge \tau_n$, where $\tau_n = o(n^{1/2})$ and $\tau_n \to \infty$ as $n\to\infty$. As before, by hypothesis, there exists a square integrable function $F$ and a constant $C$ such that
\[
k_{\pi}^{ij}(x,x)^2 \leq C V(x)^r [1 + F(\xi, \xi_i, \xi_j)]^2
\]
for each $i,j$ and $x \in S$, where, once again, $F$ can be taken to satisfy (\ref{eq:CauchyIneqF}). The augmented chain $\tilde{X}_t$ is $\tilde{V}$-uniformly ergodic
with respect to $\tilde{V}$ defined by (\ref{eq:NewLyapunov}) with the new choice of $F$, and now
\[
C_a \coloneqq \sup_{x,y\in S}\sup_{u,v \in [0,1]} \frac{|K_\pi(x,y,\xi,u,v)|}{\tilde{V}(x,u)^r + \tilde{V}(y,v)^r} < +\infty,
\]
almost surely. Defining
\begin{equation}
\label{eq:SteinCorPhi}
\varphi_n((x,u), (y,v)) = [w(x) \wedge \tau_n][w(y) \wedge \tau_n]K_{\pi}(x, y, \xi, u, v),
\end{equation}
proceeding as before, there is the same decomposition (\ref{eq:SteinNDecomp}) in
terms of $U_n$, $V_n$, and $W_n$, but now with $\varphi$ replaced by $\varphi_n$
and $W_n = n^{-1} \sum_{i=1}^n w(X_i) \wedge \tau_n$. Since $\varphi_n
\leq \tau_n \varphi_1$ and $(x,u) \mapsto \varphi_1((x,u), (x,u))$ is integrable
with respect to the product of $\nu$ and the Lebesgue measure on $[0,1]$, 
the ergodic theorem implies $\tau_n^{-1} \|V_n\|_2$ is bounded, and $n^{-1} \|V_n\|_2 \to 0$. Unlike part \ref{enu:Stein2}, $\varphi_n$ is no longer degenerate,
however, by applying \cite[Corollary 2.3]{fort2012simple} to the Hoeffding
decomposition of $U_n$, and recognising that
\[
\sup_{x, y \in S} \frac{|\varphi_n((x, u), (y, v))|}{\tilde{V}(x,u)^r + \tilde{V}(y,v)^r} \leq \tau_n C_a,
\]
it follows that there is a constant $K > 0$ depending only on $V$, the geometric ergodicity of the chain $X_t$ and the initial starting distribution (of $X_0$), such that, for $\tilde{\varphi}_n(x,y) = \int_0^1 \int_0^1 \varphi_n((x,u),(y,v)) \dd u \dd v$ and any $n=2,3,\dots,$
\begin{equation}
\label{eq:SteinUnDecay}
\mathbb{E}\left[U_n - \int \tilde{\varphi}_n(x, y) \nu(\dd x)\nu(\dd y)\right]^2
\leq \frac{K C_1^2 \tau_n^2}{n} \to 0.
\end{equation}
Denoting $E_n = \{x: \partial\pi / \partial \nu (x) \geq \tau_n\}$,
\[
\int |\varphi(x, y) - \tilde{\varphi}_n(x, y)|\nu(\dd x)\nu(\dd y)
\leq \int_{E_n\times E_n} |k_{\pi}(x, y)| \pi(\dd x)\pi(\dd y),
\]
which implies that $\int \varphi_n(x, y) \nu(\dd x) \nu(\dd y) \to 0$ by
the dominated convergence theorem, which, together with (\ref{eq:SteinUnDecay}), implies $\|U_n\|_2 \to 0$. 
Finally, concerning $W_n$, by the ergodic theorem, for any fixed $m \in \mathbb{N}$, with probability one,
\begin{multline*}
1 = \int w(x)\nu(\dd x) \geq \limsup_{n\to\infty} \frac1n \sum_{i=1}^n w(X_i) \wedge \tau_n \\\geq \liminf_{n\to\infty} \frac1n \sum_{i=1}^n w(X_i) \wedge \tau_n
\geq \int w(x)\wedge \tau_m \nu(\dd x) \to 1,
\end{multline*}
and so $W_n \cas 1$. It follows immediately that $\mathbb{S}_n(\v v^{(n)}, \v x) \cp 0$.
\end{proof}

As a simple yet useful corollary, one can also look at the case where $X_1,X_2,\dots$ are independent and identically distributed random elements on $S$, in which case, $V$ can be taken to be any arbitrary integrable function, and so Theorem \ref{thm:SteinCorrection}\ref{enu:Stein1} follows if $x\mapsto k_\pi(x,x) \in L^2(\nu)$. It is possible to weaken this assumption by using tighter estimates for the iid case, resulting in a generalization of \cite[Theorem 3.2]{Liu2017} presented in Corollary \ref{cor:IIDStein}. Analogous results can also be obtained for other types of stationary sequences using similar arguments; we refer to \cite{lee1990u} for the enabling estimates in these cases. 
\begin{corollary}
\label{cor:IIDStein}
In the setting of Theorem \ref{thm:SteinCorrection}, let $X_1,X_2,\dots$ be 
a sequence of independent random variables on $S$ with common distribution $\nu$, such that $\pi$ is absolutely continuous with respect to $\nu$.
\begin{enumerate}[label=(\alph*)]
\item \label{enu:SteinIndep1}
If $x \mapsto k_\pi(x,x) \in L^1(\nu)$, then $\mathbb{S}(\hat{\pi}^n \Vert \pi) \cp 0$ as $n \to \infty$.
\item \label{enu:SteinIndep2}
Letting $w = \partial \pi / \partial \nu$, if $x \mapsto w(x)^2 k_\pi(x,x) \in L^1(\nu)$ and \linebreak $(x,y) \mapsto w(x)w(y) k_\pi(x,y) \in L^2(\nu \times \nu)$, then $\mathbb{S}(\hat{\pi}^n \Vert \pi) = \mathcal{O}_{\mathbb{P}}(n^{-1/2})$ as $n \to \infty$.
\end{enumerate}
\end{corollary}
\begin{proof}%
As in the proof of Theorem \ref{thm:SteinCorrection}\ref{enu:Stein1}, define $\varphi_n$ according to (\ref{eq:SteinCorPhi}) and form the same decomposition $U_n, V_n, W_n$. Since $\tau_n^{-1} \|V_n\|_1$ is bounded, $n^{-1} \|V_n\|_1 \to 0$, and by the law of large numbers, $W_n \cas 1$. By \cite[Theorem 1.3.3]{lee1990u}, there exists $C > 0$ such that $\mathbb{E}[U_n - \int \tilde{\varphi}_n(x,y) \nu(\dd x)\nu(\dd y)]^2 \leq C \tau_n^2 n^{-1} \to 0$ as $n \to \infty$. The remaining arguments in the proof of Theorem \ref{thm:SteinCorrection}\ref{enu:Stein1} arrive at \ref{enu:SteinIndep1}.
For \ref{enu:SteinIndep2}, following the same arguments as Theorem \ref{thm:SteinCorrection}\ref{enu:Stein2}, this time, computing the variance of $V_n$ gives $\|V_n\|_2 = \mathcal{O}(n^{-1/2})$, while \cite[Theorem 1.3.3]{lee1990u} implies $\|U_n\|_2 = \mathcal{O}(n^{-1})$. 
\end{proof}

For the canonical Stein kernel on $\mathbb{R}^d$, if $\nabla\log p(x)$ is polynomial, then to satisfy the conditions of Theorem \ref{thm:SteinCorrection}\ref{enu:Stein1}, it suffices to establish $V$-uniform ergodicity for a drift function $V$ that grows as $V(x) \propto e^{s \|x\|}$ for some $s > 0$. In particular, inspired by \cite[Corollary 3.3]{hansen2003geometric}, the following Proposition \ref{prop:GeoErgo} provides a simple sufficient condition for Theorem \ref{thm:SteinCorrection}\ref{enu:Stein1} when the Markov chain $\{ X_k\}_{k=1}^{\infty}$ has Gaussian transition probabilities.
\begin{proposition}
\label{prop:GeoErgo}
Consider a Markov chain $\{X_k\}_{k=1}^{\infty}$ on $\mathbb{R}^d$ satisfying the recursion $X_{k+1} = \mu(X_k) + \sigma(X_k) Z_k$, where each $Z_k$ is an independent standard normal random vector. Letting $\Sigma(x) = \sigma(x) \sigma(x)^\top$, the conditions of Theorem \ref{thm:SteinCorrection}\ref{enu:Stein1} are satisfied if $k_\pi^{ij}(x,x) = \mathcal{O}_{\mathbb{P}}(\|x\|^\alpha)$ for some $\alpha > 0$ and
\begin{equation}
\label{eq:GeoErgoCondition}
\limsup_{\|x\| \to \infty} \frac{\|\mu(x)\|^2 + \mathrm{tr}(\Sigma(x))}{\|x\|^2} < 1.
\end{equation}
\end{proposition}
\begin{proof}%
Our proof resembles that of \cite[Proposition 1]{hodgkinson2021implicit}.
Let $\mathcal{P}$ denote the transition operator of $\{X_k\}_{k=1}^{\infty}$, and let $V_s(x) = e^{s\|x\|}$. It will suffice to show that $\{X_k\}_{k=1}^{\infty}$ is $V_s$-uniformly ergodic for any $s > 0$. 
By \cite[Theorem 3.1]{hansen2003geometric}, the Markov kernel $\mathcal{P}$ is irreducible with respect to Lebesgue measure and aperiodic, and has small compact sets. Fixing $s > 0$, by \cite[Theorem 15.0.1]{meyn2012markov} and \cite[Lemma 15.2.8]{meyn2012markov}, it suffices to show that $\mathcal{P}V_s(x) / V_s(x) \to 0$ as $\|x\| \to \infty$. For $Z$ a standard normal random vector, let $\alpha = \limsup_{\|x\|\to \infty} \mathbb{E}\|\mu(x) + \sigma(x) Z\|^2 / \|x\|^2$. By assumption, $\alpha < 1$, and note that by Jensen's inequality,
\begin{equation}
\label{eq:GeoErgoNormalMain}
\mathbb{E}\|\mu(x) + \sigma(x) Z\| - \|x\| \leq (\sqrt{\alpha} - 1)\|x\| \to -\infty,\quad\mbox{as }\|x\|\to\infty.
\end{equation}
Finally, \cite[Theorem 5.5]{boucheron2013concentration} implies
\[
\frac{\mathcal{P}V_s(x)}{V_s(x)} \leq \exp\left(\frac{s^2}{2} + s\mathbb{E}\|\mu(x) + \sigma(x) Z\| - s \|x\|\right),
\]
and so $\mathcal{P}V_s(x) / V_s(x) \to 0$ as $\|x\| \to \infty$ by (\ref{eq:GeoErgoNormalMain}). 
\end{proof}
Fortunately, (\ref{eq:GeoErgoCondition}) is satisfied for many algorithms generating approximate samples from distributions on $\mathbb{R}^d$, including the unadjusted Langevin algorithm \cite{roberts1996}, the (exact) implicit Langevin algorithm \cite{hodgkinson2021implicit}, and the tamed unadjusted Langevin algorithm \cite{brosse2018tamed}. In situations where a particular Markov chain of interest can itself only be approximately simulated (for example, the inexact implicit Langevin algorithm \cite{hodgkinson2021implicit}), in Corollary \ref{cor:Robust}, we appeal to the robustness of $V$-uniform ergodicity to perturbations in the total variation metric. Corollary \ref{cor:Robust} follows immediately from Theorem 1 and \cite[Theorem 1]{ferre2013regular}.

\begin{corollary}
\label{cor:Robust}
Let $X_1,X_2,\dots$ be a $V$-uniformly ergodic Markov chain on $S$
for some Lyapunov drift function $V: S \to[1,\infty)$, with corresponding
transition operator $P$. Consider a family of Markov chains $\{X_t^{\epsilon}\}_{\epsilon > 0}$ with a corresponding family of transition operators $\{P_{\epsilon}\}_{\epsilon > 0}$ such that $P_{\epsilon} \to P$ in total variation as $\epsilon \to 0^{+}$. If $V$ satisfies the hypotheses of Theorem \ref{thm:SteinCorrection}, then there exists some $\epsilon_0 > 0$ such that the results of Theorem \ref{thm:SteinCorrection} hold for the Markov chain $X_1^{\epsilon},X_2^{\epsilon},\dots$ for any $0 < \epsilon < \epsilon_0$. 
\end{corollary}

\section{Numerical experiments}
\label{sec:Numerical}

The results of two brief numerical experiments are shown. The first experiment investigates how the $\mathcal{O}_{\mathbb{P}}(n^{-1/2})$ convergence rate in KSD compares to that of MMD on a simple example where the latter can be computed explicitly. The second experiment considers a challenging target measure from a Bayesian deep learning problem, and shows that the Stein correction can still lead to observable improvements in the convergence rate under an estimated KSD. 

\subsection{Empirical study of the rate of convergence}

Despite the $\mathcal{O}_{\mathbb{P}}(n^{-1/2})$ rate of convergence in Theorem \ref{thm:SteinCorrection}\ref{enu:Stein2}, convergence rates in KSD do not necessarily carry across to other probability metrics. 
To explore the extent to which this rate of convergence might hold with other discrepancy measures, 
we conduct an empirical study with a simple
sampling problem for which the MMD, an alternative convergence determining probability metric popular in machine learning, can be computed exactly. 
Our target $\pi$ is the $d$-dimensional standard normal distribution $\mathcal{N}(\v 0, \m I_d)$ with $d = 20$, which gives $\nabla \log p(x) = -\v x$. We employ the canonical Stein kernel (\ref{eq:KCCKernel}) with IMQ and Gaussian base kernels, given by $k_{\text{IMQ}}(\v x,\v y) = (1 + \|\v x - \v y\|^2)^{-\frac12}$ and $k_{\text{Gauss}}(\v x,\v y) = e^{-\|\v x-\v y\|^2}$ respectively. We remark that the former kernel satisfies Assumption \ref{ass:KernelExplode}, while the latter does not. To sample approximately from $\mathcal{N}(\v 0, \m I_d)$, we use the tamed unadjusted Langevin algorithm (TULA), simulating
the Markov chain
\begin{equation}\label{eq:TULA}
\v X_{k+1} = \v X_k + \frac{h}{2} \cdot \frac{\nabla \log p(\v X_k)}{1 + \gamma \|\nabla \log p(\v X_k)\|} + \sqrt{h} \v Z_k,
\end{equation}
with $\v X_0 = \v 0$, and where each $\v Z_k$ is an independent standard normal random vector in $d$ dimensions,
$h > 0$ is the step size, and $\gamma > 0$ is a taming parameter, ensuring the chain can drift by no more than $h / (2\gamma)$ in Euclidean distance at each step. 
By \cite[Proposition 3]{brosse2018tamed}, for any $h,\gamma > 0$, the conditions of
Theorem \ref{thm:SteinCorrection}\ref{enu:Stein2} are satisfied. In Figure \ref{fig:MMDConvPlot}, we plot MMD to the standard normal distribution for the empirical distributions of
unadjusted and Stein-adjusted samples generated using TULA, and exact iid samples. 
As expected, the unadjusted empirical distribution fails to converge in either MMD or KSD. On the other hand, the exact and Stein-adjusted samples yield empirical distributions that converge not only at the expected $\mathcal{O}(n^{-1/2})$ rate in KSD, but also in MMD. In fact, once sufficiently many samples have been obtained, the Stein-adjusted empirical distribution using the IMQ base kernel exhibits improved performance over exact iid samples under both discrepancy measures. The Gaussian base kernel, which fails to satisfy Assumption \ref{ass:KernelExplode}, exhibits poorer performance in MMD, although does better than unadjusted samples. Altogether, this experiment supports the possibility of transferal for convergence rates from convergence determining KSD to other discrepancy measures. 
\begin{figure}[ht]
\includegraphics[width=0.465\textwidth]{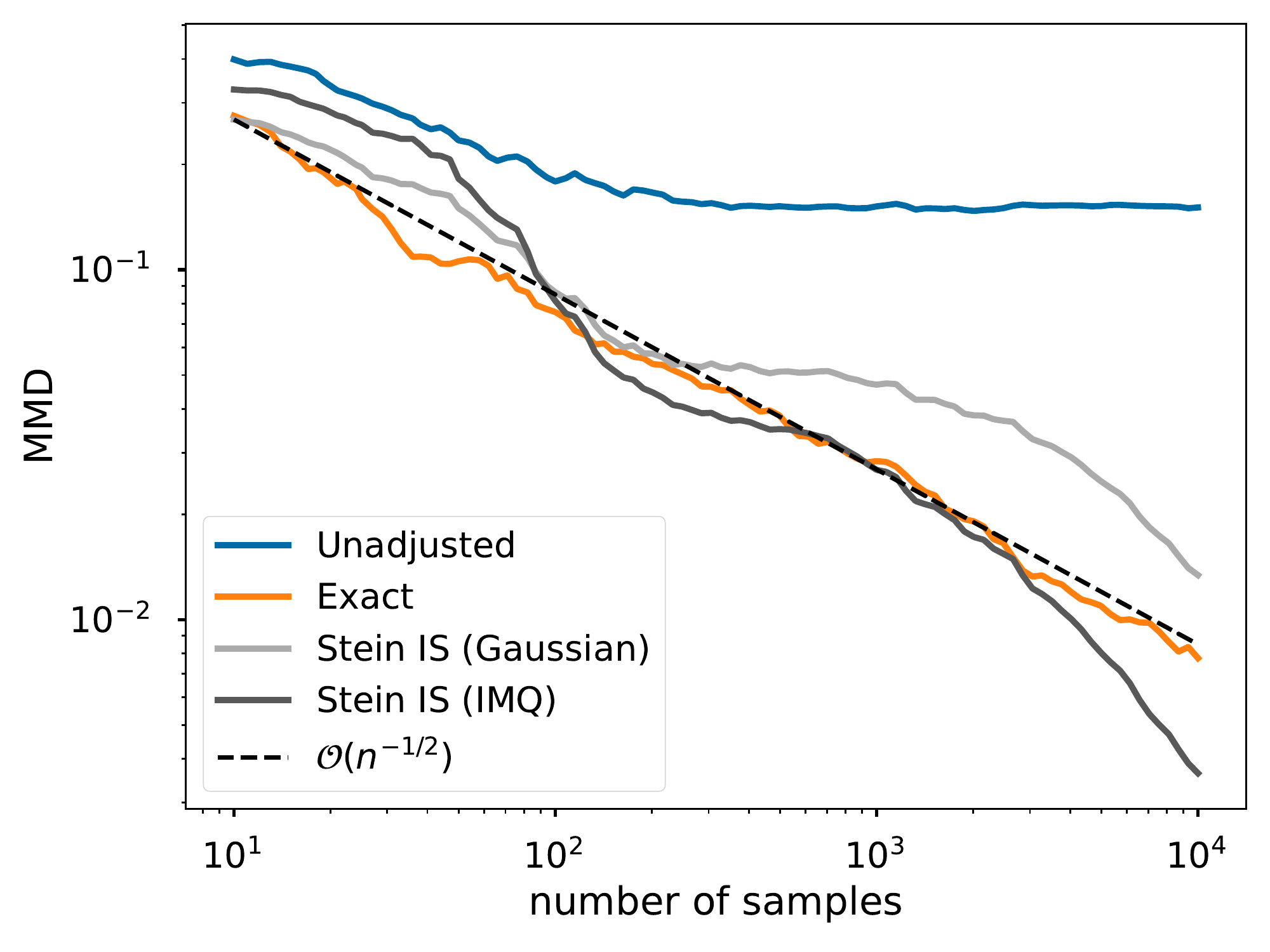}
\includegraphics[width=0.465\textwidth]{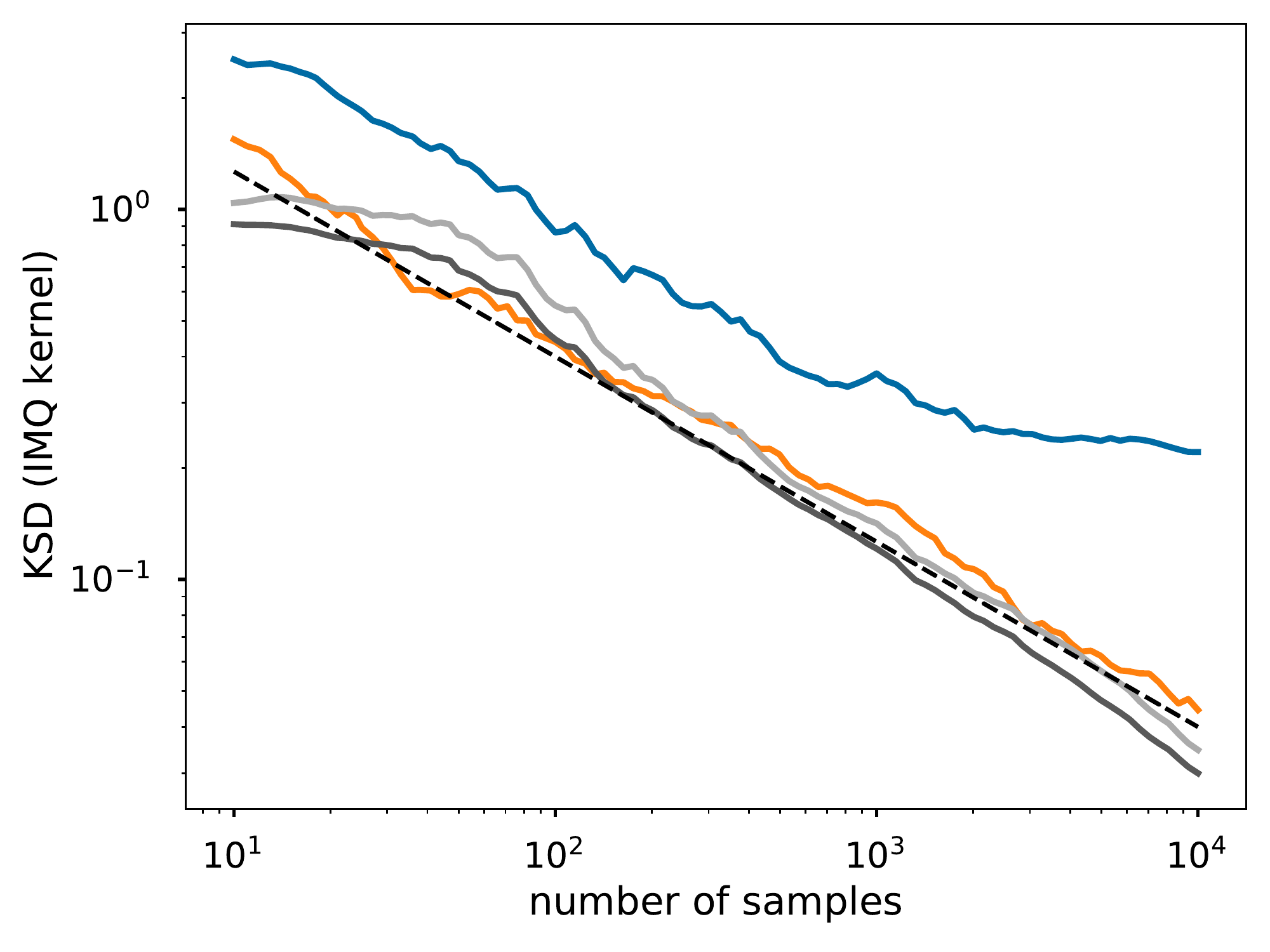}
\caption{\label{fig:MMDConvPlot}Rate of convergence in MMD to the standard multivariate normal distribution in 20 dimensions of unadjusted vs Stein-adjusted empirical distributions. Samples are obtained using TULA with step size $h = 1$ and taming parameter $\gamma = 0.05$. The corresponding Stein-adjusted weighted empirical distributions are computed with respect to the canonical Stein kernel with Gaussian/IMQ kernels and exact samples.}
\end{figure}

\subsection{Comparison with subsampled and full-data KSD}

The second example considers Bayesian logistic regression with a Gaussian prior, yielding a target measure with log-density 
\begin{equation*}
\label{eq:LogRegTarget}
\log p(\v x) \varpropto  -\sum_{i=1}^{n} \left(\log\left(1 + \exp(\v a_{i}^{\top} \v x)\right) - b_{i} \v a_{i}^{\top} \v x \right) + \frac{\lambda}{2}||\v x||^2.
\end{equation*}

Again, the TULA algorithm is used for sampling, but with subsampled gradient (that is, $\log p$ terms are replaced by their subsampled counterpart, $\log p^{\mathscr{S}_t}$). Note that the application of Stein importance sampling to subsampled gradients can be justified by viewing the problem through the lens of random switching Markov chains \cite{cloez2015exponential}. Indeed, TULA with subsampled gradients is a random switching Markov chain indexed over the subsample $\mathscr{S}_t$. For any fixed subsample $\mathscr{S}$, the Markov chain $X^{\mathscr{S}}$ induced by the TULA algorithm with target density $p^{\mathscr{S}}$ is $V_{\mathscr{S}}$-uniformly ergodic, where $V_{\mathscr{S}}(x) = e^{a_{\mathscr{S}} (1+\|x\|^2)^{1/2}}$ \cite[Proposition 3]{brosse2018tamed}. Taking $a = \min_{\mathscr{S}} a_{\mathscr{S}}$, a modification of \cite[Lemma 15.2.9]{meyn2012markov} implies that $X^\mathscr{S}$ is $V$-uniformly ergodic, where $V(x) = e^{a(1+\|x\|^2)^{1/2}}$. Since $\mathscr{S}$ is arbitrary, \cite[Lemma 2.8]{cloez2015exponential} implies that TULA with subsampled gradients is $V^\alpha$-uniformly ergodic for some $\alpha > 0$, and hence satisfies the conditions of Theorem \ref{thm:SteinCorrection}(b).

Two of the larger data sets from the binary regression examples in \cite{ong2018gaussian} are considered, namely {\sf krkp} ($n=3196 , d=38$) and {\sf spam} $(n=4601, d = 105)$. In the former example, the step size and taming parameter are taken to be $h = 0.1$ and $\gamma = 0.05$, respectively, with $h = 0.05$ and $\gamma = 0.01$ in the latter example.
To coincide with the use of subsampled gradients for sampling, we also consider performance of Stein importance sampling under a subsampled Stein kernel. For the conditions of Theorem \ref{thm:SteinCorrection} to hold, it is essential that the subsampled gradients in the Stein kernel are independent of those used in the sampling proceess. Therefore, there are two subsampling-related parameters in this experiment: the number $n_s$ of subsampled data points at each iteration of TULA, and the number $n_k$ of subsampled data points in computing the subsampled Stein kernel. For simplicity, we take $n_s = n_k$, and set $n_s = 500$ and $n_s = 1000$ for {\sf krkp} and {\sf spam}, respectively. 

Results are shown in Figure  \ref{fig:LogRegResults}, where we plot the (full-data) KSD for the unadjusted samples, the samples obtained with the full-data Stein correction, and the samples obtained with the subsampled Stein correction. We observe in both examples the full-data Stein correction achieving an approximate $\mathcal{O}(n^{-1/2})$ rate of convergence toward the end of the simulation. We observe a similar phenomenon for the subsampled Stein correction in the {\sf krkp} example, but not necessarily for the {\sf spam} example, which is still exhibiting prelimiting behaviour.

\begin{figure}[ht]
\includegraphics[width=0.465\textwidth]{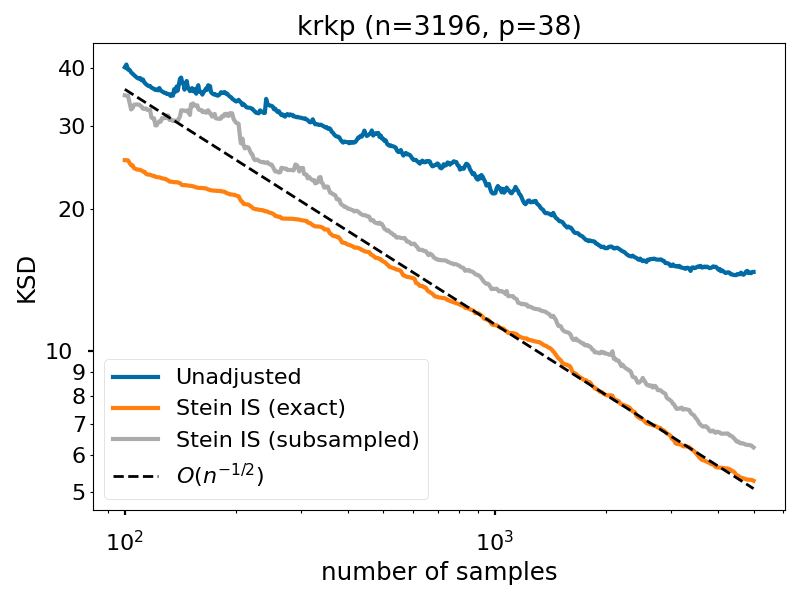}
\includegraphics[width=0.465\textwidth]{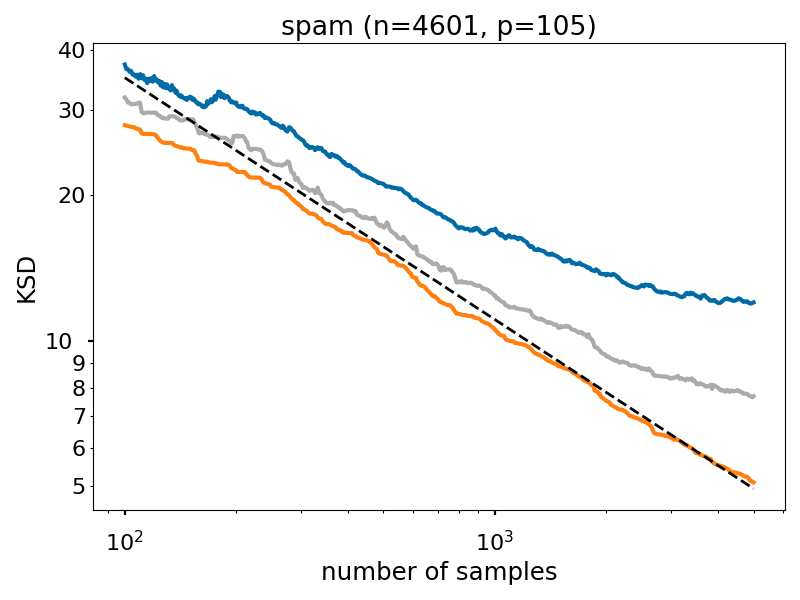}
\caption{Results for the two binary regression problems {\sf krkp} and {\sf spam}. All KSD values reported are computed with respect to the full data set.}\label{fig:LogRegResults}
\end{figure}

\section*{Acknowledgements}
Liam Hodgkinson and Fred Roosta are affiliated with the
International Computer Science Institute, Berkeley, CA, 94704, USA. All authors are supported in part by the Australian Centre of Excellence for Mathematical and
Statistical Frontiers (ACEMS), under Australian Research Council grant CE140100049.
Fred Roosta is supported by DARPA D3M (FA8750-17-2-0122), Cray/77000, and ARC DECRA (DE180100923). We are grateful to Samuel Power for drawing our attention to Zanella processes.

\end{document}